\documentclass[11pt,notitlepage,a4paper]{article}

\usepackage{amssymb} 
\usepackage[intlimits]{amsmath}
\usepackage{amsfonts}
\usepackage{amsthm} 

\usepackage{mathptmx}

\usepackage{hyperref}
\hypersetup{colorlinks=true,linkcolor=RoyalPurple,citecolor=RoyalPurple}

\usepackage{cite}

\usepackage{graphicx}

\usepackage{geometry}

\usepackage{color}
\usepackage[dvipsnames]{xcolor}

\let\oldsqrt\sqrt
\def\sqrt{\mathpalette\DHLhksqrt}
\def\DHLhksqrt#1#2{%
\setbox0=\hbox{$#1\oldsqrt{#2\,}$}\dimen0=\ht0
\advance\dimen0-0.2\ht0
\setbox2=\hbox{\vrule height\ht0 depth -\dimen0}%
{\box0\lower0.4pt\box2}}

\allowdisplaybreaks

\newcommand{\R}{\mathbb{R}} 
\newcommand{\N}{\mathbb{N}} 

\newcommand{\supp}{\textnormal{supp}} 

\renewcommand{\phi}{\varphi}

\newcommand{\cE}{{\mathcal E}}

\newcommand{\cG}{{\mathcal G}}
\newcommand{\cH}{{\mathcal H}}

\newcommand{\cL}{{\mathcal L}}

\newcommand{\Ds}{{(-\Delta)}^s}

\newcommand{\dn}{D}
\newcommand{\ch}{\mathcal H}

\theoremstyle{definition}
\newtheorem{defi}{Definition}[section]
\newtheorem{remark}[defi]{Remark}

\theoremstyle{plain} 
\newtheorem{thm}[defi]{Theorem}
\newtheorem{prop}[defi]{Proposition}
\newtheorem{lemma}[defi]{Lemma}
\newtheorem{cor}[defi]{Corollary}

\theoremstyle{definition}

\numberwithin{equation}{section} 

\title{
Integral representation of solutions to higher-order fractional Dirichlet problems on balls
}
\author{
Nicola Abatangelo\footnote{D\'{e}partement de math\'{e}matique, Universit\'{e} Libre de Bruxelles CP 214, Boulevard du Triomphe, 1050 Ixelles, Belgium, nicola.abatangelo@ulb.ac.be}, 
\ Sven Jarohs\footnote{Institut f\"ur Mathematik, Goethe-Universit\"at, Frankfurt, Robert-Mayer-Stra\ss e 10, D-60054 Frankfurt, Germany, jarohs@math.uni-frankfurt.de}, and
Alberto Salda\~{n}a\footnote{Institut f\"ur Analysis, Karlsruhe Institute for Technology, Englerstra\ss e 2, 76131, Karlsruhe, Germany, \newline alberto.saldana@partner.kit.edu}
}
\date{}

\begin{document}

\maketitle

\begin{abstract}
We provide closed formulas for (unique) solutions of nonhomogeneous Dirichlet problems on balls involving any positive power $s>0$ of the Laplacian.  We are able to prescribe values outside the domain and boundary data of different orders using explicit Poisson-type kernels and a new notion of higher-order boundary operator, which recovers normal derivatives if $s\in\N$. Our results unify and generalize previous approaches in the study of polyharmonic operators and fractional Laplacians.  As applications, we show a novel characterization of $s$-harmonic functions in terms of Martin kernels, a higher-order fractional Hopf Lemma, and examples of positive and sign-changing Green functions. 

\end{abstract}

\section{Introduction}

The study of explicit solutions and representation formulas for differential operators is a very classical and important problem in PDEs, which dates back to \cite{P1820,green} for the Laplace operator. In general, this kind of expressions are powerful tools to obtain a wide range of qualitative properties\textemdash symmetry, a priori bounds, regularity\textemdash and precise quantitative estimates for solutions and its derivatives.  In this work, we study such explicit expressions involving any positive power of the Laplacian, \emph{i.e.,} for $(-\Delta)^s$ with $s>0$ in the unitary ball $B$ complemented with suitable (Dirichlet) boundary conditions. This paper complements and extends our previous work on higher-order fractional Laplacians \cite{AJS16a,AJS16b} (an extended preprint can be found in \cite{AJS16}, see also \cite{ADFJS18} for the half-space case).

One of the main difficulties in this setting is that for $s\not\in \N$ the fractional Laplacian $(-\Delta)^s$ becomes a nonlocal operator.
The pointwise notion of the $s$-Laplacian is a subtle issue, see Remark \ref{pointwise:rmk} below. For our main results, we use the following definition: for $N\in\N$, $m\in\N_0$, $\sigma\in(0,1)$, and $s=m+\sigma$, we write $(-\Delta)^s u$ to denote $(-\Delta)^m(-\Delta)^\sigma u$, where 
\begin{equation}\label{fractional}
(-\Delta)^{\sigma}u(x)=c_{N,\sigma} \lim_{\epsilon\to 0^+}\int_{\R^N\backslash\overline{B}_\epsilon(0)} \frac{u(x)-u(x+y)}{|y|^{N+2\sigma}}\ dy\qquad \text{ for }\  x\in \R^N,
\end{equation}
and $c_{N,\sigma}$ is a positive normalization constant, see \eqref{cnsigma}. To guarantee integrability in \eqref{fractional}, we require a growth condition at infinity encoded in the space
\begin{align*}
 \cL^1_{\sigma}:=\left\{u\in L^1_{loc}(\R^N)\;:\; \|u\|_{\cL^1_{\sigma}}<\infty \right\}, \qquad \text{ where }\quad \|u\|_{\cL^1_{\sigma}}:=\int_{\R^N}\frac{|u(x)|}{1+|x|^{N+2\sigma}}\ dx,
\end{align*}
see  \cite{FW16,S07,BB99}.  Thus, $(-\Delta)^\sigma u$ is well defined in $B$ if, for example, $u\in C^{2\sigma+\alpha}(B)\cap\cL_\sigma^1$ for some $\alpha>0$. For the physical and mathematical importance of $(-\Delta)^s$ we refer to \cite{AJS16a} and the references therein.

As studied in \cite{B1905,BGR61,GGS10,DG2016,AJS16b}, the Green function for $(-\Delta)^s$ in $B$ is given by Boggio's formula
\begin{equation}\label{green}
G_s(x,y)\ =\ k_{N,s} {|x-y|}^{2s-N}\int_0^{\rho(x,y)}\frac{t^{s-1}}{(t+1)^{\frac{N}{2}}}\ dt,
\qquad s>0,\ \ \ x,y\in \R^N,\ \ \ x\neq y,
\end{equation}
where $\rho(x,y)=\delta(x)\delta(y)|x-y|^{-2}$ for $x,y\in \R^N,$ $x\neq y$, $\delta(x):=(1-|x|^2)_+$, and $k_{N,s}$ is a positive normalization constant given in \eqref{c} below. 

Using \eqref{green}, one can obtain unique representation formulas for suitable functions, for example, an element $u\in C^{2,0}(B)\cap C^{1,0}(\overline{B})$ (see Section \ref{Notation} below to clarify notation) is represented by
\begin{equation}\label{poisson-green-classical}
u(x)=\int_{B} G_1(x,y)(-\Delta) u(y)\;dy- \int_{\partial B} \partial_{\nu}G_1(x,z)\,u(z)\;dz \quad\text{ for $x\in B$.}
\end{equation}
The term $-\partial_{\nu}G_1$ is called the \emph{Poisson kernel} (for $-\Delta$ in $B$) \cite{P1820} and $\partial_\nu$ denotes the normal derivative. This kernel can also be used to obtain  the \emph{harmonic extension} of a given boundary condition (b.c.) $g\in C(\partial B)$, that is, the function
\begin{align*}
u(x)=- \int_{\partial B}\partial_{\nu}G_1(x,z)\,g(z)\;dz\qquad \text{ for } x\in B
\end{align*}
is a solution of $-\Delta u = 0$ in $B$ and $u=g$ on $\partial B$.

The notion of Poisson kernel can be extended to other operators, for example, to the polyharmonic operator $(-\Delta)^m$, $m\in\N$, or to the fractional Laplacian $(-\Delta)^s$, $s\in(0,1)$, where similarities but also fundamental differences appear. Consider, for instance, the \emph{bilaplacian} $\Delta^2=(-\Delta)\circ(-\Delta)$.  As a higher-order operator, additional b.c. are required in order to obtain well-posedness of linear problems; for instance, the equation $\Delta^2 u = f\in C(\overline{B})$ has a unique solution $u\in C^{4,0}(\overline{B})$ if the \emph{Dirichlet} b.c. $u=\partial_\nu u=0$ on $\partial B$ is assumed.  Furthermore, for $h_0\in C^{1,0}(\partial B)$ and $h_1\in C(\partial B)$, the \emph{biharmonic extension} \cite[p. 141]{GGS10} is given by
\begin{equation}\label{ubihar}
u(x)=\int_{\partial B} K(x,z) h_0(z)\ dz + \int_{\partial B}L(x,z) h_1(z)\ dz\qquad \text{ for }\ x\in B,
\end{equation}
where $K(x,z):=\partial_\nu|_{z}\Delta G_2(x,\cdot)$ and $L(x,z):=-\Delta G_2(x,\cdot)(z)$ are usually also called Poisson kernels.  In this case, $u$ given by \eqref{ubihar} satisfies $\Delta^2 u=0$ in $B$, $u=h_0$ on $\partial B$, and $\partial_\nu u = h_1$ on $\partial B$. In the general case $s\in\N$, b.c. of the type
\begin{align}\label{kbc}
 \partial_\nu^k u = h_k\qquad \text{ on }\partial B\quad \text{ for }k\in\{0,1,\ldots,s-1\},
\end{align}
can be prescribed for $h_k\in C^{s-k,0}(\partial B)$. For details, see \cite{E75-1,E75-2,GGS10}.

For the fractional Laplacian $(-\Delta)^\sigma$ with $\sigma\in(0,1)$, due to the nonlocality, well-posedness of linear problems is achieved by prescribing b.c. in the whole complement of the ball $\R^N\backslash\overline{B}$. A Poisson kernel for this notion of b.c. can be found in \cite[equation (3), Chapitre V]{R37}, see also \cite{L72,B99,B16,FW16}. More precisely, if $g\in C(\R^N)\cap L^\infty(\R^N)$, then the \emph{$\sigma$-harmonic extension} $u:\R^N\to \R$ of $g$ in $B$ is given by 
\begin{align}\label{P:rep:sigma}
u(x)=\chi_{\overline{B}}(x)\int_{\R^N\backslash\overline{B}} \Gamma_\sigma(x,y)g(y)\;dy  + \chi_{\R^N\backslash\overline{B}}(x)g(x)\qquad\text{ for }\ x\in\R^N,
\end{align}
where $\Gamma_\sigma(x,\cdot):=-(-\Delta)^\sigma G_\sigma(x,\cdot)$ in $\mathbb R^N\backslash\overline{B}$ for $x\in \R^N$.  The function $u$ given by \eqref{P:rep:sigma} is the unique continuous $\sigma$-harmonic extension of $g$.  However, in this nonlocal case, 
it is possible to prescribe another type of b.c.: a \emph{singular trace}.  As a consequence, a $\sigma$-harmonic function $u$ is represented, for $x\in B$, via 
\begin{align}\label{singular}
u(x)=\chi_{\overline{B}}(x)\int_{\R^N\setminus\overline{B}}\Gamma_\sigma(x,y) u(y)\;dy
\ +\ 
m_\sigma\int_{\partial B}M_\sigma(x,z)v(z)\;dz,\qquad v(z)=\lim_{B\ni y\to z}\frac{u(y)}{(1-|y|^2)^{\sigma-1}},
\end{align}
where $m_\sigma>0$ is given in \eqref{c} and $M_\sigma$ is known as a \emph{Martin kernel} \cite{M41,
MS99,B99,nicola}, which, in the case of a ball, has the following explicit formula \cite{B99,nicola,AJS16b} for all 
$s>0$ 
\begin{equation}\label{martin}
M_s(x,z)\ :=\lim_{B\ni y\to z}\frac{G_s(x,y)}{{(1-|y|^2)}^{s}}
=\frac{k_{N,s}}{s}\,\frac{{(1-|x|^2)}_+^s}{{|x-z|}^N},
\quad x\in\R^N,\ z\in\partial B,\ x\neq z, 
\end{equation}
where $k_{N,s}>0$ is given in \eqref{c} below.  A function $u$ given by \eqref{singular} is sometimes called a \emph{large solution} or a \emph{boundary blow-up solution}, because of its singular behavior at $\partial B$.  Although the Martin kernel $M_\sigma$ can be used to prescribe a (singular) boundary condition and it converges (pointwisely) to the Poisson kernel for the Laplacian as $\sigma\to 1$, it is usually not called a Poisson kernel; this name is reserved for $\Gamma_\sigma$, which relates to values in $\R^N\backslash\overline{B}$, see Remark \ref{conv:rem}.  We also note that, for $s\in(0,1)$, abstract representation formulas for more general domains are available \cite{nicola,B99}, but the kernels are rarely explicit. 

\medskip

In this paper, we aim to combine these two settings\textemdash the nonlocal case, $s\in(0,1)$, and the classical higher-order case, $s\in\N$\textemdash to better understand the nature of higher-order operators in general. Previous results in the higher-order setting have considered b.c. of the type 
\begin{align}\label{lh}
 u=0\quad \text{ in }\R^N\backslash\overline{B}\qquad \text{ and }\qquad\lim_{B\ni x\to z}\frac{u(x)}{(1-|x|^2)^{s-1}}=\varphi(z)\quad \text{ for all }z\in\partial B.
\end{align}
for suitable $\varphi$.  Regularity, existence, and uniqueness of solutions of $(-\Delta)^su=f$ in $B$ satisfying \eqref{lh} are studied (in a more general setting) in \cite{G15:2}.  Similar boundary operators to \eqref{lh} (with $s$ instead of $s-1$) were also used in \cite{RS15} to show an integration-by-parts formula and a Pohozaev identity.  This paper is the first to study \emph{explicit} solutions for $(-\Delta)^su=f$ in $B$ satisfying a general notion of b.c. (see \eqref{dn-vs-partial} below), which generalizes \eqref{kbc} and \eqref{lh}.

\medskip

Our first result studies the \emph{Poisson kernel} for $(-\Delta)^s$ in $B$, which prescribes values in $\R^N\backslash\overline{B}$. Let $\sigma\in(0,1)$, $m\in\N_{0}$, $s=m+\sigma$, and 
\begin{equation}\label{PK}
\Gamma_s(x,y)\ :=\ (-1)^m\frac{\gamma_{N,\sigma}}{{|x-y|}^N}\frac{(1-|x|^2)_+^s}{(|y|^2-1)^s}\qquad \text{ for }x\in \mathbb R^N,\ y\in \R^N\backslash \overline{B},
\end{equation}
where $\gamma_{N,\sigma}$ is a positive normalization constant given in \eqref{c} below.
The kernel \eqref{PK} was previously known to be a Poisson kernel only for $s\in(0,1)$, see \eqref{P:rep:sigma}. The following Theorem shows that $\Gamma_s$ is a Poisson kernel for any $s>0$, if integrability is guaranteed. To include the case $s\in \N$ in some of our next results we also define $\Gamma_s\equiv 0$ if $s\in \N$ and to ease notation we write $B_r$ to denote the ball of radius $r>0$ centered at zero.  
\begin{thm}[Poisson kernel]\label{poisson:thm}
Let $N\in\N$, $m\in\N_0$, $\sigma\in(0,1)$, $s=m+\sigma$, $G_s$ as in \eqref{green}, and $\Gamma_s$ as in \eqref{PK}. Then, 
\begin{align}\label{GsDsG}
\Gamma_s(x,y)=\ -(-\Delta)_y^m(-\Delta)_y^\sigma G_s(x,y)\qquad \text{for $x\in \mathbb R^N$, $y\in \R^N\setminus\overline{B}$} 
\end{align}
and, if $\psi\in \cL_\sigma^1$ with $\psi=0$ in $B_{r}$, $r>1$,
and $u:\R^N\to \R$ is given by 
\begin{equation}\label{u:def}
u(x)\ =\int_{\R^N\setminus\overline{B}}\Gamma_s(x,y)\psi(y)\;dy\ +\ \psi(x),
 \end{equation}
then $u\in C^{\infty}(B)\cap C^s(B_{r})\cap H^s(B_{\rho})$ for any $\rho\in(1,r)$ and $u$ is the unique pointwise solution in the space $C^s(\overline{B})\cap H^s(B)$ of 
\begin{equation}\label{u:prob}
(-\Delta)^s u = 0\quad \text{ in }B\qquad \text{ with }\qquad u=\psi\quad \text{ on }\R^N\backslash\overline{B}.
\end{equation}
\end{thm}
The proof of Theorem \ref{poisson:thm} is made by induction, using a new recurrence equation that relates $\Gamma_s$ with $\Gamma_{s-1}$, see Proposition~\ref{Gamma:it}.  In the following, as it is customary in the fractional setting, we use the name \emph{Poisson kernel} only for the kernel $\Gamma_s$.

If $\psi\not\in\cL_\sigma^1$, then, in general, it is \emph{not} possible to compute pointwisely $(-\Delta)^m(-\Delta)^\sigma u$, with $u$ as in \eqref{u:def}, see Remark~\ref{pointwise:rmk}.  On the other hand, notice that $\Gamma_s(x,\cdot)$ has a strong singularity at the boundary $\partial B$ and, because of this, we require that $\psi=0$ near $\partial B$ in Theorem \ref{poisson:thm}.  For functions which are different from zero near $\partial B$, the Poisson kernel for $(-\Delta)^s$ is more involved: it is the sum of the Poisson kernel $\Gamma_\sigma$ for $(-\Delta)^\sigma$ and suitable (boundary) correction terms, see Corollary \ref{I-2} and Theorem \ref{uniqueness:thm:2} below. To describe these corrector terms and a more general family of solutions with boundary values we introduce first a new notion of \emph{higher-order trace operator}.
\begin{defi}
For $N\in\N$, $k\in\N_0,$ $\sigma\in(0,1],$ and $x\in \R^N$, let 
\begin{equation}\label{dn-vs-partial}
\dn^{k+\sigma-1}u(z):=\frac{(-1)^k}{k!}\lim_{x\to z}\frac{\partial^k}{\partial (|x|^2)^k}[(1-|x|^2)^{1-\sigma} u(x)]\qquad \text{ for }z\in\partial B.
\end{equation}
Here and in the following, all limits are taken in the normal direction with respect to $\partial B$ from inside $B$, that is, the limit $\lim\limits_{x\to z}$ is always meant for $x\in B$ such that $\frac{x}{|x|}=z\in \partial B$.  
\end{defi}

We remark that the traces $\dn^{k+\sigma-1}$ are \emph{not} fractional differential operators in general (observe that the composition $D^s\circ D^t$ is not well-defined). The constants and the derivatives with respect to $|x|^2$ in \eqref{dn-vs-partial} are very useful to simplify expressions; however, analogous results can be obtained using, for example, standard normal derivatives $\partial_\nu$, see Remark \ref{K:Equiv}. We refer to Subsection \ref{remarks-derivative} below for the main properties of $\dn^{k+\sigma-1}$.
Trace operators combining weights and derivatives were also used in \cite[Theorem 6.1]{G15:2} (see also \cite{grubb16,G15:3}), where solvability in a more general setting is studied.

The traces $\dn^{k+\sigma-1}$ unify several previous approaches. For instance, $\dn^{\sigma-1}$ with $\sigma\in(0,1)$ is the \emph{singular trace} used in \eqref{singular} and, under suitable assumptions (see Lemma \ref{prop:de} below), $\dn^{k+\sigma-1}$ can be reduced to the boundary operator in \eqref{lh}. Of particular importance is the case $\sigma=1$, where $\dn^{k+\sigma-1}=D^k$ is (up to a constant) the differential operator $(\frac{\partial}{\partial |x|^2})^k$. These derivatives were used in \cite{E75-1, E75-2} together with explicit boundary kernels to study closed formulas for polyharmonic Dirichlet problems. Using \eqref{dn-vs-partial}, we extend the results in \cite{E75-1, E75-2} to the higher-order fractional setting. 

\begin{defi}
For $N\in\N$, $\sigma\in(0,1]$, $m\in\N_0$, $s=m+\sigma$, $k\in\{0,1,\ldots,m\}$, $x,y\in \R^N,$ $x\neq y$, and $\theta\in \partial B$, the (fractional) \emph{Edenhofer kernels} are given by
\begin{align}\label{sEK}
E_{k,s}(x,\theta):=\frac{1}{\omega_N}(1-|x|^2)_+^{s}\dn^{m-k} \zeta_x(\theta),\qquad \text{ where }\quad \zeta_x(y):=\frac{|y|^{N-2}}{|x-y|^N}
\end{align}
and $\omega_N:=|\partial B|$.
\end{defi}
See Remark \ref{noconstants} for the original formulation in \cite{E75-1, E75-2}. The following theorem uses \eqref{PK} and \eqref{sEK} to provide explicit solutions to nonhomogeneous linear problems with  one nonlocal condition and $m+1$ prescribed boundary traces. Let 
\begin{align}\label{delta}
\delta(x):=(1-|x|^2)_+\qquad \text{for }x\in\R^N.
\end{align}

\begin{thm}[Explicit solution]\label{main:thm}
Let $m\in\N_0$, $\sigma,\alpha\in(0,1]$, $s=m+\sigma$, $2s+\alpha\notin \N$, $g_k\in C^{m-k,0}(\partial B)$ for $k=0,\ldots,m$, $f\in C^\alpha(\overline{B})$, $h\in \cL^1_{\sigma}$ such that $h=0$ in $B_{r}$, $r>1$, and $u:\R^N\to\R$ be given by 
$u(x)=h(x)$ for $x\in\R^N\backslash\overline{B}$ and 
\begin{align*}
u(x)=\int_B G_s(x,y)f(y)\ dy+\int_{\R^N\backslash\overline{B}}\Gamma_s(x,y)h(y)\ dy+
\sum_{k=0}^{m}\ \int_{\partial B}E_{k,s}(x,\theta)\ g_k(\theta)\ d\theta\qquad \text{for $x\in B$}.
\end{align*}
Then, $u\in C^{2s+\alpha}(B)$, $\delta^{1-\sigma}u\in C^{m,0}(\overline{B})$ and, for $k=0,1,\ldots,m,$
\begin{align*}
(-\Delta)^s u=f\quad \text{ in }\ \ B,\qquad u=h\quad \text{ on }\ \ \mathbb R^N\backslash\overline{B},\qquad \text{ and }\quad \dn^{k+\sigma-1} u= g_k\quad \text{ on }\ \ \partial B.
\end{align*}
\end{thm}
The proof is based on the results from \cite{E75-1, E75-2} and the following extraordinary fact: 
if $v$ is a suitable $(m+1)$-harmonic function in $B$ and $\delta$ as in \eqref{delta}, then $u=\delta^{\sigma-1} v$ is an $(m+\sigma)$-harmonic function in $B$, that is,
\begin{align*}
(-\Delta)^{m+1}v=0 \quad\text{ in }B\quad \text{ implies }\quad (-\Delta)^{m+\sigma}(\delta^{\sigma-1}v)=0 \quad \text{ in $B$};
\end{align*}
see also Corollary \ref{Almansi:cor} below for a more general and precise statement.  This relationship seems to be previously not known 
and relies on properties of Martin kernels and a variant of \emph{Almansi's formula}, 
which decomposes an $m$-harmonic function into a finite sum of harmonic functions with polynomial coefficients, see \eqref{Almansi:f} and \eqref{A:d}.  For a discussion on the convergence of these kernels and solutions as $\sigma\to 1$, we refer to Remark \ref{conv:rem}.

\medskip

In a particular set of functions, the solution given by Theorem \ref{main:thm} is unique, which yields the following integral representation formula. 
\begin{thm}[Integral representation formula]\label{uniqueness:thm}
Let $m\in\N_0$, $\sigma,\alpha\in(0,1]$, $r>1$, $s=m+\sigma$, $2s+\alpha\not\in\N$, $u\in \cL_{\sigma}^1\cap C^{2s+\alpha}(B)$ be such that 
\begin{align*}
\delta^{1-\sigma} u \in C^{ m+\alpha}(\overline{B}),\qquad (-\Delta)^s u\in C^\alpha(\overline{B}),\qquad \text{ and }\qquad u=0\ \ \text{ in }B_{r}\backslash\overline{B}.
\end{align*}
Then, for $x\in B$,
\begin{align*}
u(x)=\int_B G_s(x,y)(-\Delta)^s u(y)\ dy + \int_{\R^N\backslash\overline{B}}\Gamma_s(x,y)u(y)\ dy+ \sum_{k=0}^{m}\ \int_{\partial B}E_{k,s}(x,\theta)\dn^{k+\sigma-1} u(\theta)\ d\theta.
\end{align*}
\end{thm}
The proof uses the theory of weak solutions as developed in \cite{AJS16a} and on regularity in H\"{o}lder and fractional Sobolev spaces, see Subsection \ref{reg:sec}; in particular, the results from \cite{G15:2} play an important role in optimizing the assumptions.  The requirement $u=0$ in $B_{r}\backslash\overline{B}$, $r>1$, is used to ensure integrability with $\Gamma_s$; however, this assumption can be removed at the expense of extra regularity hypothesis, as the next results shows.
\begin{thm}\label{uniqueness:thm:2}
Let $m\in\N_0$, $\sigma,\alpha\in(0,1]$, $r>1$, $s=m+\sigma$, $2s+\alpha\not\in\N$, $u\in \cL_{\sigma}^1\cap C^{2s+\alpha}(B)\cap L^{\infty}(B_{r}\backslash\overline{B})$, $(-\Delta)^s u\in C^\alpha(\overline{B}),$ and 
\begin{align}\label{reg:assu}
(-\Delta)^\sigma u\in C^{m-\sigma-1+\beta}(\overline{B}),\qquad \text{where}\quad \beta>0\quad \text{ is such that }\quad m-\sigma-1+\beta>0.
\end{align}
For $x\in B$, set $w(x):=u(x)-\int_{\R^N\backslash\overline{B}}\Gamma_\sigma(x,y)u(y)\ dy$. Then
\begin{align}\label{w:reg}
\delta^{1-\sigma} w\in C^{m+\alpha}(\overline{B})
\end{align}
and, for $x\in B$,
\begin{align*}
u(x)=\int_B G_s(x,y)(-\Delta)^s u(y)\ dy + \int_{\R^N\backslash\overline{B}}\Gamma_\sigma(x,y)u(y)\ dy&+ \sum_{k=0}^{m}\ \int_{\partial B}E_{k,s}(x,\theta)\dn^{k+\sigma-1} w(\theta)\ d\theta,
\end{align*}
where $E_{k,s}$, $\Gamma_\sigma$ are given in \eqref{sEK}, \eqref{PK} respectively. Moreover, if $\sigma\in(0,1)$, then $\dn^{\sigma-1} w=\dn^{\sigma-1} u$ at $\partial B$ and, for $k\in\{1,2,\ldots,m\}$,
\begin{align*}
\dn^{k+\sigma-1} w(\theta)=\gamma_{N,\sigma}\frac{(-1)^k}{k!}\lim_{x\to\theta} \frac{\partial^{k-1}}{\partial (|x|^2)^{k-1}}\int_{\R^N\backslash\overline{B}}\frac{u(x)-u(y)}{|x-y|^N\ (|y|^2-1)^\sigma}\ dy,\qquad \theta\in\partial B.
\end{align*}
\end{thm}
Assumption \eqref{reg:assu} is used to guarantee \eqref{w:reg}, see Lemma \ref{dnG:l} below. This allows to construct suitable corrector terms enabling the use of $\Gamma_\sigma$ instead of $\Gamma_s$.  Property \eqref{w:reg} may also follow from regularity assumptions on $u$ in $B_{r}\backslash\overline{B}$, but this is not entirely trivial and to keep this paper short we do not pursue this here and we only comment on the particular case $s=1+\sigma$ in Remark \ref{e:s:1sigma}.

As a consequence of our approach, we can provide a characterization for suitable $s$-harmonic functions in terms of {Martin kernels}. This yields in particular a very simple way to construct $(s+t)$-harmonic functions. The following corollary generalizes \cite[Remark 6.16.3]{AJS16}. 
Recall the definition of Martin kernel $M_s$ given in \eqref{martin}. 

\begin{cor}\label{Almansi:cor} Let $\alpha,\sigma\in(0,1]$, $m\in\N_0$, $s=m+\sigma$, $2s+\alpha\not\in\N$, and $u\in C^{2s+\alpha}(B)$ be such that $(-\Delta)^s u=0$ in $B$, $\delta^{1-\sigma}u\in C^{ m+\alpha}(\overline{B})$, and $u=0$ in $\R^N\backslash\overline{B}$.  Then $u\in C^{\infty}(B)$ and there are unique functions $g_k\in C(\partial B)$ such that
  \begin{align*}
u(x)=\sum_{k=0}^{m}\int_{\partial B} M_{k+\sigma}(x,\theta)\ g_k(\theta)\ d\theta.
  \end{align*}
  In particular, if $u\in C^{2s+\alpha}(B)$, $\delta^{1-\sigma}u\in C^{ m+\alpha}(\overline{B})$, and $u=0$ in $\R^N\backslash\overline{B}$ then 
\begin{align*}
(-\Delta)^{s}u=0 \quad\text{ in B}\qquad \text{ implies }\qquad (-\Delta)^{s+t}(\delta^{t}u)=0 \quad \text{ in $B$}\qquad \text{ for }t>m-s=-\sigma.
\end{align*}
\end{cor}

We remark that the Edenhofer kernels $E_{k,s}$ are also related to higher-order traces of the Green function $G_s$.  In particular, for $E_{m-1,s}$ and $E_m$, we have the following.
\begin{lemma}\label{explicit-edenhofer}
For $m\in\N_0$, $\sigma\in(0,1]$, $s=m+\sigma$, $x\in B,$ and $z\in \partial B$ we have that
\begin{equation}
E_{m,s}(x,z)=m_s\dn^s[G_s(x,\cdot)](z)=\frac{1}{\omega_N}\frac{\delta(x)^s}{|x-z|^N}=m_s M_s(x,z)\label{comp1a}
\end{equation}	
and, if $m\geq 1$,
\begin{align}
E_{m-1,s}(x,z)&=m_{s-1}\ \dn^{s-1} ( -\Delta G_{s}(x,\cdot))(z)=\frac{1}{4\omega_N}\frac{\delta(x)^{s}}{|x-z|^{N+2}}(N\delta(x)-(N-4)|x-z|^2),\label{comp1}\\
E_{m,s}(x,z)&=m_{s-1}\ \dn^{s-2}(\Delta G_{s}(x,\cdot))(z).\label{complement}
\end{align}
\end{lemma}
	
Lemma \ref{explicit-edenhofer} can be used to obtain a \emph{higher-order fractional Hopf-Lemma} for the homogeneous Dirichlet problem (see \cite[Theorem 5.7]{GGS10} for the case $s\in \N$).
\begin{cor}[Hopf Lemma]\label{hopf:lemma}
Let $s>1$, $\alpha\in(0,1)$, $2s+\alpha\not\in\N$, $f\in C^{\alpha}(\overline{B})\backslash\{0\}$ be nonnegative, and let $u\in \cH_0^s(B)$ be the unique weak solution of
\begin{align*}
(-\Delta)^{s} u=f\gneq 0\quad \text{ in }B\qquad\text{ with }\qquad u=0\quad \text{ on }\R^N\backslash\overline{B}.
\end{align*}
Then, $u\in C^{2s+\alpha}(B)\cap C^{s}(\overline{B})$ and, for $z\in\partial B$,
\begin{align} 
\dn^s u(z)&=\frac{m_{s-1}}{m_s}\dn^{s-2} \Delta u(z)=\int_B M_s(y,z)f(y)\ dy> 0,\nonumber\\
\dn^{s-1}(- \Delta) u(z)&=\frac{s-1}{4k_{N,s-1}}\int_B \frac{\delta(y)^{s}}{|y-z|^{N+2}}(N\delta(y)-(N-4)|y-z|^2)f(y)\ dy.\label{s:eq}
\end{align}
In particular, $\dn^{s-1} (-\Delta) u>0$ at $\partial B$ if $N\leq 4$.
\end{cor}

To close this introduction, we discuss some implications that the previous results have on positivity preserving properties for $(-\Delta)^s$ with $s>1$. In \cite[Theorem 1.1]{AJS16a} we showed that maximum principles fail for $(-\Delta)^s$ in disjoint sets if $s=m+\sigma$ with $m\in \N$ \emph{odd} and $\sigma\in(0,1)$, i.e. there is a positive function $f$, a disjoint set $\Omega\subset \R^N$, and a unique (weak) solution $u\in \cH^s_0(\Omega)$ such that $(-\Delta)^su=f$, but $u$ is sign-changing (see Theorem \ref{mp-fails} below).  

Using the Poisson kernel $\Gamma_s$ and Theorem \ref{poisson:thm} we complement this result by showing that, if $m$ is \emph{even}, then the Green function of two disjoint balls is \emph{positive} inside the domain $\Omega$.  In particular, this implies that the maximum principle \emph{does} hold for two (arbitrarily far away) disjoint balls if $m$ is even.  Moreover, the following result provides an alternative proof of the fact that maximum principles fail if $m$ is odd, since it guarantees that the Green function changes sign in this case.

To state the result we introduce some notation: let $\Omega(t):=B\cup B^t$, where $t>2$, $B^t:=B_1(te_1)$, denote by $G_{\Omega(t)}$ the Green function of $(-\Delta)^s$ in $\Omega(t)$ (see Proposition \ref{existenceuniqueness} for existence), and set
\begin{align}\label{T0:def}
T_0:=2+ |B|\gamma_{N,\sigma}=2+\left(\frac{2\sin(\pi\sigma)}{N\,\pi}\right)^{1/N}\in(2,3).
\end{align}
\begin{thm}\label{even:thm2}
Let $N\in \mathbb N$, $m\in\mathbb N_0$, $\sigma\in(0,1)$, $s=m+\sigma$, and $t>T_0$. Then 
\begin{align}\label{inB}
G_{\Omega(t)}&>0\qquad \text{ in }\quad \{(x,y)\in(B\times B)\cup (B^t\times B^t)\::\: x \neq y\},\\
G_{\Omega(t)}&>0\qquad \text{ in }\quad (B\times B^t)\cup (B^t\times B),\qquad \text{ if }\quad m \text{ is even}\label{niBe},\\
G_{\Omega(t)}&<0\qquad \text{ in }\quad (B\times B^t)\cup (B^t\times B),\qquad \text{ if }\quad m \text{ is odd}\label{niBo}.
\end{align}
\end{thm}

We do not expect that maximum principles hold in general for $m$ even, but a counterexample in this case is still missing, see \cite{AJS16a} for further discussions in this regard.

Finally, we show that maximum principles also fail for $(-\Delta)^s$ in \emph{connected domains} if $s=m+\sigma$ with $m\in\N$ odd, by connecting two disjoint balls with a thin tube and using a perturbation argument. This theorem complements our previous result \cite[Theorem 1.1]{AJS16a} (see Theorem \ref{mp-fails} below).  For the definition of weak solutions, see Subsection \ref{Notation} below.

\begin{thm}\label{connected:cor}
	Let $N\geq 2$, $m\in \N$ be an odd number, $\sigma\in(0,1)$, and $s=m+\sigma$. Moreover, let $\Omega=B_1(0)\cup B_1(3e_1)$, $L:=\{t e_1\::\: 0<t<3\}$, and 
	\begin{align*}
	\Omega_n=\Omega \cup \{\ x\in\R^N\::\: \operatorname{dist}(x\ ,\ L)<\frac{1}{n}\ \}\qquad \text{ for } n\in\N.
	\end{align*}
	There is $n\in\N$, a nonnegative function $f_n\in L^\infty(\Omega_n)$, and a weak solution $u_n\in\cH_0^{s}(\Omega_n)$ of $(-\Delta)^{s}u_n=f_n\geq 0$ in $\Omega_n$, $u=0$ on $\R^N\backslash\overline{B}$,
	such that $\operatorname{essinf}_{\Omega_n} u_n<0$ and $\operatorname{esssup}_{\Omega_n} u_n>0$.
\end{thm}

The paper is organized as follows: in Subsection \ref{Notation} we specify the notation and conventions that are used in the rest of the paper, in Section \ref{Auxiliary:sec} we collect a series of useful remarks and results of independent interest regarding the higher-order trace operator, integration by parts, regularity, and a proof of (a variant of) Almansi's formula Lemma \ref{Almansi:l}. The proof of Theorem \ref{poisson:thm} is contained in Section \ref{sectionproof1} together with our discussion on maximum principles, where the proofs of Theorems \ref{even:thm2} and \ref{connected:cor} are given.
In Section \ref{IRF:sec} we focus our attention to explicit solutions and integral representation formulas with boundary kernels and it includes the proof of Theorems \ref{main:thm}, \ref{uniqueness:thm}, \ref{uniqueness:thm:2}, Lemma \ref{explicit-edenhofer}, and Corollaries \ref{Almansi:cor} and \ref{hopf:lemma}.

\subsection{Notations}\label{Notation}

Let $N\in \N$ and $U,D\subset \R^N$ be nonempty measurable sets. We denote by $\chi_U: \R^N \to \R$ the characteristic function and by $|U|$ the Lebesgue measure of $U$.  We fix $B:=B_1(0)$ and $B_r:=B_r(0)$ for $r>0$. For any $s\in\R$, we define as usual $H^s(\R^N):=\left\{u\in L^2(\R^N)\;:\; (1+|\xi|^{2})^{\frac{s}{2}}\ \widehat u\in L^2(\R^N)\right\},$ where $\widehat u$ denotes the Fourier transform. Moreover, if $U$ is open, $m\in \N_0$, $\sigma\in[0,1)$, and $s=m+\sigma$, we define $\cH^{s}_0(U)$ as 
\begin{align*}
\cH^{s}_0(U)&:=\{u\in H^{s}(\R^N)\;:\; u= 0\;\text{on $\R^N\setminus U$}\}
\end{align*}
equipped with the norm 
$\|u\|_{\cH^s_0(U)}:=(\sum_{|\alpha|\leq m}\|\partial^{\alpha} u\|_{L^2(U)}^2+\cE_{s}(u,u))^{\frac{1}{2}}$, 
where
\begin{equation}\label{scalar:p}
\cE_{s}(u,v):=\left\{\begin{aligned}
&\cE_{\sigma}((-\Delta)^{\frac{m}{2}}  u,(-\Delta)^{\frac{m}{2}} v),&& \quad \text{if $m$ is even,}\\
&\sum_{k=1}^{N}\cE_{\sigma}(\partial_k (-\Delta)^{\frac{m-1}{2}} u,\partial_k (-\Delta)^{\frac{m-1}{2}} v),&&
\quad  \text{if $m$ is odd,}
\end{aligned}\right.
\end{equation}
for $u,v\in \cH^{s}_0(U)$ (see \cite{AJS16a,AJS16}), where $\cE_{0}(u,v)=(u,v)_{L^2(\R^N)}$ and, for $\sigma\in(0,1)$,
\begin{align}\label{cnsigma}
\cE_{\sigma}(u,v):=\frac{c_{N,\sigma}}{2}\int_{\R^N}\int_{\R^N}\frac{(u(x)-u(y))(v(x)-v(y))}{|x-y|^{N+2\sigma}}\ dxdy,%
\quad\ c_{N,\sigma}:=\frac{4^{\sigma}\Gamma(\frac{N}{2}+\sigma)}{\pi^{\frac{N}{2}}|\Gamma(-\sigma)|}.
\end{align}
 Let $U$ be a Lipschitz open set. We say that $u\in\cH_0^s(U)$ is a weak solution of $(-\Delta)^s u=f$ in $U$ with $u=0$ in $\R^N\backslash U$ if 
\begin{align*}
\cE_{s}(u,\varphi)=\int_B f(x)  \varphi(x)\ dx\qquad \text{ for all }\varphi \in \cH_0^s(U).
\end{align*}
If $U$ has Lipschitz boundary, we put $H^s(U):=\{u \chi_{U}\;:\; u\in H^s(\R^N)\}$. Note that, in general, $H^s(U)\neq \cH^s_0(U)$.

\medskip

For $m\in \N_0$ and $U$ open we write $C^{m,0}(U)$ to denote the space of $m$-times continuously differentiable functions in $U$ and, for $\sigma\in(0,1]$ and $s=m+\sigma$, we write $C^s(U):=C^{m,\sigma}(U)$ to denote the space of functions in $C^{m,0}(U)$ whose derivatives of order $m$ are (locally) $\sigma$-H\"older continuous in $U$ or (locally) Lipschitz continuous in $U$ if $\sigma=1$. We denote by $C^s(\overline{U})$ (if $\partial U$ is smooth enough) the set of functions $u\in C^s(U)$ such that 
\begin{align}\label{Hn}
\|u\|_{C^s(U)}:=\sum_{|\alpha|\leq m}\|\partial^{\alpha}u\|_{L^{\infty}(U)}+\sum_{|\alpha|=m}\ \sup_{\substack{x,y\in U \\x\neq y}}\  \frac{|\partial^{\alpha}u(x)-\partial^{\alpha}u(y)|}{|x-y|^{\sigma}}<\infty.
\end{align}
Moreover, for $s\in(0,\infty]$, $C^s_c(U):=\{u\in C^s(\R^N): \supp \ u\subset\subset U\}$ and $C^s_0(U):=\{u\in C^s(\R^N): u= 0 \text{ on $\R^N\setminus U$}\}$, where $\supp\ u:=\overline{\{ x\in U\;:\; u(x)\neq 0\}}$ is the support of $u$.  
\medskip

We use $u^+:=u_+:= \max\{u,0\}$ and $u^-:=-\min\{u,0\}$ to denote positive and negative part of $u$ respectively.  If $f: \R^N\times\R^N\to \R$ we write $(-\Delta)_x^s f(x,y)$ to denote derivatives with respect to $x$, whenever they exist in some appropriate sense. 
Whenever it is meant in the pointwise sense, we write $(-\Delta)^s u$ to denote $(-\Delta)^m(-\Delta)^\sigma u$, where the fractional Laplacian $(-\Delta)^\sigma u$ is evaluated pointwisely as in \eqref{fractional}.
We sometimes write $\partial_\nu$ instead of $\partial_r$ or $\frac{\partial}{\partial r}$, but these notations refer to the differential operator $\frac{x}{|x|}\cdot \nabla u$ evaluated at $\partial B$.
\medskip

Let $\omega_N:=|\partial B|=2\pi^{\frac{N}{2}}\ \Gamma(\frac{N}{2})^{-1},$ where $\Gamma$ denotes the usual \emph{Gamma function}. We frequently use the following normalization constants
\begin{align}\label{c}
k_{N,s}&:=\frac{2^{1-2s}}{\omega_N{\Gamma(s)}^{2}},\quad m_{s}:={\frac{2\, s\, k_{N,1}}{k_{N,s}}}
= \frac{\Gamma(s)\Gamma(s+1)}{2^{1-2s}},\quad \gamma_{N,\sigma}:=\frac{2}{ \Gamma(\sigma)\,\Gamma(1-\sigma)\omega_N}.
\end{align}

For $x\in \R^N$, $s>0$ and a suitable $u:\R^N\to\R$, we denote 
\begin{equation}\label{GHHs}
 \begin{aligned}
 \cG_s u(x)&:=\int_B G_{s}(x,y)(-\Delta)^{s} u(y)\ dy,\\
 \cH u(x)&:=\int_{\partial B}2M_1(x,z) u(z)\ dy,\\
\cH_s u(x)&:=\left\{\begin{aligned}
&\int_{\R^N\backslash\overline{B}}\Gamma_s(x,y) u(y)\ dy+\chi_{\mathbb R^N\backslash\overline{B}}(x)u(x)&& \quad \text{for}\ s\in (0,\infty)\backslash\N,\\
&0,&&
\quad  \text{for $s\in\N$,}
\end{aligned}\right.
\end{aligned}
\end{equation}
where $G_s$ is as in \eqref{green}, $\Gamma_s$ is as in \eqref{PK}, and $M_1$ is as in \eqref{martin} with $s=1$. As usual, in dimension one ($N=1$), the boundary integral is meant in the sense $\int_{\partial B}f(\theta)\ d\theta=f(-1)+f(1)$.  Finally, we set $\delta(x):=(1-|x|^2)_+$ and, with a slight abuse of notation, we simply write $\delta^{-\alpha},\ \alpha>0,$ to denote the function $x\mapsto \delta(x)^{-\alpha}$ for $x\in B$ and $x\mapsto 0$ for $x\in \R^N\backslash\overline{B}$.

\section{Preliminary results}\label{Auxiliary:sec}
 
In this section we collect a series of remarks and results of independent interest that we use in the following sections for the proofs of our main theorems. 
 
\subsection{Higher-order trace operators}\label{remarks-derivative}
We begin with some remarks on the trace operators $\dn^{k+\sigma-1}$ defined in \eqref{dn-vs-partial}. Observe that, by definition, if $m\in \N$ and $u\in C^{m,0}(B)$, then
\begin{align}\label{page10}
\dn^{m}u(x)=\frac{(-1)^m}{m!}\frac{\partial^m}{\partial (|x|^2)^m}\ u(x)=\frac{(-1)^m}{m!}\left(\frac{1}{2|x|}\frac{x}{|x|}\cdot\nabla\right)^m u(x), \qquad x\in B,
\end{align}
where the factor $\frac{(-1)^m}{m!}$ is a normalization constant used to simplify calculations. In particular, note that the following variant of Leibniz's rule holds
 \begin{align*}
\dn^{m+\sigma-1}(uv)=\sum_{k=0}^{m}\dn^{k+\sigma-1}u\ \dn^{m-k}v\qquad \text{for suitable functions $u$ and $v$.}
\end{align*}
In general, $\lim\limits_{x\to z}\frac{\partial^m}{\partial (|x|^2)^m}\ {u}\neq \lim\limits_{x\to z}(\frac{1}{2}\partial_\nu)^m\ {u}$; indeed, let $m=2$ and ${u}$ be a smooth function, 
then for $x\in B$ and $r=|x|$,
\begin{align*}
{\frac{\partial^2}{\partial (r^2)^2}\ u=
\left(\frac{1}{2r}\partial_r\right)\left(\frac{1}{2r}\partial_r\right)u=
\frac{1}{4r}\left(-\frac1{r^{2}}\partial_r u+\frac1r\partial_{rr} u\right)}\qquad \text{ in }B,
\end{align*}
therefore $\lim\limits_{x\to z}\frac{\partial^2}{\partial (|x|^2)^2}\ {u(x)}=\frac{1}{4}(- \partial_{\nu} {u}+(\partial_{\nu})^2{u})(z)\neq \frac{1}{4}(\partial_{\nu})^2{u}(z)$ for $z\in\partial B$ if $\partial_\nu w\neq 0$ at $\partial B$.  Nevertheless, if the traces $\dn^{k+\sigma-1} u$ are known at $\partial B$ for $k=0,\ldots,m$ and $\sigma\in(0,1]$, then one can \emph{infer} the values of 
$\lim\limits_{x\to z}(\partial_\nu)^k[\delta(x)^{1-\sigma}u(x)]$ at $\partial B$ and vice-versa. See also Remark \ref{K:Equiv} below.
Moreover, note that $\lim\limits_{r\to 1}(\frac{\partial}{\partial r})^m(1-r)^m = (-1)^m\, m!$, and therefore
\begin{align*}
\dn^{m+\sigma-1}\delta^{m+\sigma-1}(\theta)=1\qquad \text{ and }\qquad \dn^{m+\sigma-1}\delta^{k+\sigma-1}(\theta)=0\quad \text{ for } k\in \N_0,\ k\neq m,\ \theta\in \partial B.
\end{align*}
A very useful property of the trace $\dn^{k+\sigma-1}$ is given in the following lemma.
\begin{lemma}\label{prop:de}
If $u:\R^N\to \R$ is such that $\delta^{1-\sigma} u\in C^{m,0}(\overline{B})$ and $\dn^{k+\sigma-1} u =0$ at $\partial B$ for $k= 0,\ldots,m-1$, then 
\begin{align}
\dn^{m+\sigma-1}u(z) =\lim_{x\to z}\frac{u(x)}{\delta(x)^{m+\sigma-1}} \qquad \text{ for }\ z\in\partial B.\label{p:4}
\end{align}
In particular, for $f\in C^{\alpha}(\overline{B})$, $\alpha>0$,
\begin{align}  
\dn^{k+\sigma-1}\Big(\int_B G_{k+\sigma-1}(x,\cdot)\,f(x)\;dx\Big)(z)=\int_B M_{k+\sigma-1}(x,z) f(x)\ dx\quad  \text{ for }x\in B,\ z\in\partial B.\label{p:5} \end{align}
\end{lemma}
\begin{proof}
Set $v=\delta^{1-\sigma}u$ and $u$ as in the statement. Since $\frac{\partial^k}{\partial (|x|^2)^k}{v}(z)=0$ for $k=0,\ldots,m-1$, we have that $\Big(\frac{\partial}{\partial |x|}\Big)^kv(z)=0$ for $k=0,\ldots,m-1$ and therefore 
$\frac{\partial^m}{\partial (|x|^2)^m}{v}(z)=\frac{1}{2^m}\Big(\frac{\partial}{\partial |x|}\Big)^m v(z)$. For $t\in(0,1)$, let $f(t):={v}((1-t)z)$.  Then, using a Taylor expansion at 0,
\begin{align*}
f(t)=\sum_{k=0}^m \frac{1}{k!}f^{(k)}(0)\ t^k+o(t^m)=\frac{1}{m!}\frac{\partial^m}{\partial (|x|)^m}{v}(z)\ (-t)^m+o(t^m)\quad\text{ as }t\to 0.
\end{align*}
But then
\begin{align*}
\frac{(-1)^m}{m!}\frac{\partial^m}{\partial (|x|^2)^m}{v}(z)=\frac{(-1)^m}{ m!\ 2^m}\frac{\partial^m}{\partial (|x|)^m}{v}(z) = \lim_{t\to 0} \frac{{v}((1-t)z)}{(2-t)^mt^m}
=\lim_{t\to 0} \frac{{v}((1-t)z)}{\delta((1-t)z)^m}=  \lim_{x\to z} \frac{{v}(x)}{\delta(x)^m},
\end{align*}
which shows \eqref{p:4}. Finally, let $x,z\in \R^N$ with $x\neq z$ and $\rho(x,z)=\delta(x)\delta(z)|x-z|^{-2}$.  Using a change of variables {$t=\rho(x,z)\tau$ in formula \eqref{green} }we have that
\begin{align}\label{Gs:decomp}
G_s(x,y)=k_{N,s} \delta(x)^s\ \int_0^1 \frac{\tau^{s-1} \delta(y)^s}{( \delta(x) \delta(y)\tau +|x-y|^2)^{\frac{N}{2}}}\ d\tau.
\end{align}	
Then, for $f\in C^\alpha(\overline{B})$, the function $u(y):=\int_B G_{k+\sigma-1}(x,y) f(x)\ dx$ satisfies that $\delta^{1-\sigma} u\in C^{k{,0}}(\overline{B})$ and $\dn^{i+\sigma-1} u =0$ at $\partial B$ for $i=0,\ldots,k-1$, thus \eqref{p:5} follows as a consequence of \eqref{p:4} and \eqref{martin}.
\end{proof}

\subsection{A variant of Almansi's formula in balls}

The following is a useful decomposition of polyharmonic functions. Its proof relies in the so-called \emph{Almansi's formula}, see \cite{A1899,N35}, which decomposes an $m$-harmonic function into a finite sum of harmonic functions with polynomial coefficients of the type $|x|^{2k}$ with $k\leq m$.  Recall that $\cH$ denotes the standard harmonic extension, see \eqref{GHHs}.

\begin{lemma}\label{Almansi:l}
Let $m\in \N_0$ and 
\begin{equation}\label{v12}
v\in C^{2m+2,0}(B)\cap C^{m{,0}}(\overline{B})\qquad \text{ satisfy }\qquad \Delta^{m+1}v=0\qquad \text{ in } B.
\end{equation}
There are unique functions $h_k\in C(\partial B),$ $k=0,\ldots,m$, such that 
\begin{align}\label{Almansi:f}
v(x)=\sum_{k=0}^m\delta^{k}(x)\ch h_k(x)=\sum_{k=0}^m m_{k+1} \int_{\partial B}M_{k+1}(x,\theta)\ h_k(\theta)\ d\theta\qquad \text{ for }x\in B.
\end{align}
\end{lemma}
\begin{proof}
We follow closely \cite[Proposition 1.3]{ACL83}. Let $v$ as in \eqref{v12}. We show first that there are unique $h_k\in C(\partial B),$ $k=0,\ldots,m$, such that 
\begin{align}\label{A:d}
v(x)=\sum_{k=0}^m|x|^{2k}\ch h_k(x)\qquad \text{ for }x\in B.
\end{align}
We argue by induction on $m$.  For $m=0$ we have that $v\in C^{2,0}(B)\cap C(\overline{B})$ and $\Delta v=0$ in $B$.  Since $v\in C(\overline{B})$ we have that $v(x)=\int_{\partial B} 2M_1(x,\theta) v(\theta)\ d\theta = \ch v(x)$ for $x\in B$, by Green's representation. Hence \eqref{A:d} holds with $h_0=v$ and the uniqueness follows from the maximum principle. Now, let $m\in\mathbb N$ and assume that \eqref{A:d} holds for any function $w\in C^{2m,0}(B)\cap C^{m-1{,0}}(\overline{B})$ satisfying $\Delta^{m}w=0$ in $B$. Let $v$ {as in \eqref{v12}}. 
Since $\Delta v$ is $m$-harmonic, there are unique $g_k\in C(\partial B),$ $k=0,\ldots,m$ such that 
\begin{align}\label{Lap:v}
\Delta v(x) = \sum_{k=0}^{m-1}|x|^{2k}\ch g_k(x)\qquad \text{ for }x\in B.
\end{align}
For $k=1,\ldots,m$, let $\widetilde g_{k-1}:= \ch g_{k-1}$ and $h_k\in C^{2{,0}}(B)\cap C(\overline{B})$ be given by 
\begin{align}\label{hk:def}
h_k(x):=\frac{1}{4k}\int_0^1 \widetilde g_{k-1}(\tau x)\tau^{k-2+\frac{N}{2}}\ d\tau\qquad \text{ for }x\in B. 
\end{align}
Note that $h_k$ is the unique continuous solution (\emph{cf.}  \cite[Th\'{e}or\`{e}me 1]{N35}) of 
\begin{align*}
4k(k-1+\frac{N}{2})h_k(x)+4k\ |x|\frac{\partial}{\partial |x|}h_k(x) = \widetilde g_{k-1}(x)\qquad \text{ for }x\in B,\ \ k=1,\ldots,m.
\end{align*}
By \eqref{hk:def} we have, for $k=1,\ldots,m$, that $\Delta h_k =0$ in $B$, $h_k=\ch h_k$ (because $h_k\in C(\overline{B})$) and, by a simple calculation, $\Delta(|x|^{2k}h_k(x)) = |x|^{2(k-1)}\widetilde g_{k-1}(x)$ for $x\in B$.  Then, by setting $h_0(x):=v(x)-\sum_{k=1}^{m}|x|^{2k}h_k(x)$ for $x\in B$, we obtain that
\begin{align*}
v(x)=\sum_{k=0}^{m}|x|^{2k}h_k(x)\qquad \text{ for }x\in B.
\end{align*}
Furthermore, by \eqref{Lap:v}, $\Delta h_0 = \Delta v - \sum_{k=1}^{m}\Delta(|x|^{2k}h_k(x))= \Delta v - \sum_{k=0}^{m-1}|x|^{2k}\widetilde g_k(x) =0$ in $B$ and $h_0=\ch h_0$ in $B$ since $h_0\in C(\overline{B})$. Since the uniqueness of $h_0$ can be inferred from the uniqueness of $h_k$ with $k\geq 1$, we have that \eqref{A:d} holds for any $m\in\N$. 

Similar arguments also yield that if $\varphi(x):=\sum_{k=0}^{m-1}|x|^{2k}\ch h_k(x)$ for $x\in B$ and $h_k\in C(\partial B),$ $k=0,\ldots,m-1$, then $\varphi$ is $m$-harmonic, see \cite[Proposition 1.2]{ACL83}.
\medskip

We now show that \eqref{Almansi:f} follows from \eqref{A:d} and the binomial theorem
\begin{align}\label{binom}
\delta(x)^m=(1-|x|^2)^{m}=|x|^{2m}+\sum_{k=0}^{m-1}\binom{m}{k}(-1)^k|x|^{2k}.
\end{align}
We argue again by induction on $m$. For $m=0$, formula \eqref{Almansi:f} is true since $v(x)=\delta(x)^0\ch v(x)$. Assume \eqref{Almansi:f} holds for $m$ and let $v$ {as in \eqref{v12}}.
Then, by \eqref{A:d} and \eqref{binom}, there are unique $g_k\in C(\partial B),$ $k=0,\ldots,m$, such that
\begin{align*}
 v(x)&=|x|^{2m}\ch g_m(x)+ \sum_{k=0}^{m-1}|x|^{2k}\ch g_k(x)\\
& =\delta(x)^{m}\ch g_m(x)+ \sum_{k=0}^{m-1}|x|^{2k}(\ch g_k(x)-\binom{m}{k}(-1)^k\ch g_m(x)).
\end{align*}
Since $x\mapsto \sum_{k=0}^{m-1}|x|^{2k}\ch(g_k(x)-\binom{m}{k}(-1)^k g_m(x))$ is $m$-harmonic, by \cite[Proposition 1.2]{ACL83}, the induction hypothesis yields that there are unique $h_k\in C(\partial B)$ such that
\begin{align*}
 v(x)=\delta(x)^{m}\ch g_m(x)+ \sum_{k=0}^{m-1}\delta(x)^{k}\ch h_k(x),
\end{align*}
that is, \eqref{Almansi:f} holds for $m+1$, and the proof is finished.
\end{proof}

\begin{remark}
Almansi's formula holds in a more general setting, see \cite[Proposition 1.3]{ACL83}; in particular, it can be used to describe $m$-harmonic functions in star-shaped domains.
\end{remark}

\subsection{Integration by parts}

The next lemma is an integration by parts formula for functions which do not belong to $H^s(\R^N)$ (such as $\delta^{s-1}$ for $s>0$). Furthermore, it also states that for functions in $\cH_0^s(\Omega)$, the notions of weak and distributional solutions coincide.

\begin{lemma}\label{ibyp}
Let $\Omega\subset \R^N$ be an open bounded set with Lipschitz boundary, $\alpha,\sigma\in(0,1)$, $m\in\N_0$, $s=m+\sigma$, $u\in C^{2s+\alpha}(\Omega)\cap\cL^1_\sigma$. Then
\begin{align*}
\int_{\R^N} u\, (-\Delta)^s \varphi \ dx = \int_{\R^N} \varphi\,(-\Delta)^{m}(-\Delta)^\sigma u\ dx \qquad \text{for all $\varphi\in C^\infty_c(\Omega)$.}
\end{align*}
Moreover, if $u\in \cH_0^s(\Omega)$ then 
\begin{align}\label{lclaim}
\int_{\R^N} u\, (-\Delta)^s\varphi\ dx = \cE_s(u,\varphi)\qquad \text{ for all }\varphi\in C_c^\infty(\Omega).
\end{align}
\end{lemma}
\begin{proof}
	Let $\varphi\in C^\infty_c(\Omega)$ be such that $\varphi\in C^\infty_c(U)$, where $U$ is a smooth subset of $\Omega$. By assumption, we have $v:=(-\Delta)^m\varphi\in  C^\infty_c({U})$, $(-\Delta)^s \varphi\ =\ (-\Delta)^\sigma v\in C^\infty(\R^N)\cap H^{2s}(\R^N)$ (see \cite[Remark 3.2]{AJS16}, \cite[Proposition~2.6]{S07}, \cite{AJS17b}), and $(-\Delta)^\sigma u\in C^{2m+\alpha}({U})$, because $u\in C^{2s+\alpha}({U})$ (see the proof of \cite[Propositions 2.5, 2.6, and 2.7]{S07}).
	Let $G_\sigma^{{U}}$ and $\Gamma_\sigma^{{U}}$ denote respectively the Green function and Poisson kernel associated to $(-\Delta)^\sigma$ in ${U}$ (for existence, see \emph{e.g.} \cite{B99,nicola}).  Then, by \cite[Proposition 2, equation (10)]{nicola}, $v(x)=\int_{U} G_\sigma^{{U}}(x,y)\,(-\Delta)^\sigma v(y)\;dy$ and
	\begin{align*}
	u(x)&=\int_{U} G_\sigma^{{U}}(x,y)\,(-\Delta)^\sigma u(y)\;dy+\int_{\R^N\setminus {U}}\Gamma_\sigma^{{U}}(x,y)\,u(y)\;dy\qquad  \text{for $x\in {U}$.}
	\end{align*}
	 Note that, by \cite[Proposition 2]{nicola}, we have $\Gamma_\sigma^{U}(x,y)=-(-\Delta)_y^\sigma G_\sigma^{U}(x,y)$ for $x\in {U}$ and $y\in\R^N\backslash \overline{{U}}$. These facts, Fubini's theorem, and a standard interchange of derivatives and integral (due to the dominated convergence theorem), imply that
	\begin{align*}
	&\int_{{U}}u(x)\,(-\Delta)^\sigma v(x)\;dx\ = \\
	&= \int_{{U}}\Bigg[\int_{{U}} G_\sigma^{{U}}(x,y)\,(-\Delta)^\sigma u(y)\;dy +\int_{\R^N\setminus {U}}\Gamma_\sigma^{{U}}(x,y)\,u(y)\;dy\Bigg](-\Delta)^\sigma v(x)\;dx\\
	&= \int_{{U}}\int_{{U}} G_\sigma^{{U}}(x,y)\,(-\Delta)^\sigma v(x)\;dx\ (-\Delta)^\sigma u(y)\;dy\ + \int_{\R^N\setminus {U}}u(y)\int_{{U}}\Gamma_\sigma^{{U}}(x,y)(-\Delta)^\sigma v(x)\;dxdy \\
	&= \int_{{U}} v(y)\,(-\Delta)^\sigma u(y)\;dy\ -\ \int_{\R^N\setminus {U}}u(y)\,(-\Delta)^\sigma\Bigg[\int_{{U}}G_\sigma^{{U}}(x,\cdot)(-\Delta)^\sigma v(x)\;dx\Bigg](y)\;dy \\
	&= \int_{{U}} v(y)\,(-\Delta)^\sigma u(y)\;dy -
	\int_{\R^N\setminus {U}}u(y)\,(-\Delta)^\sigma v(y)\;dy.
	\end{align*}
	Therefore,
	\begin{align*}
	\int_{\R^N}u\,(-\Delta)^s\varphi\;dx= \int_{\R^N} (-\Delta)^\sigma u\,(-\Delta)^m\varphi\;dx=\int_{\R^N} (-\Delta)^m(-\Delta)^\sigma u\,\varphi\;dx,
	\end{align*}
	by classical integration by parts.  Finally, \eqref{lclaim} follows using the Fourier transform, see \cite[Lemma B.4]{AJS16b,AJS16}, where only the ball is considered, but the same proof carries the case of open bounded Lipschitz sets.
\end{proof}

Using Lemma \ref{ibyp} we deduce the following important property of Martin kernels.

\begin{prop}\label{lem:sharm-intro}
Let $s>0$, $g\in C(\partial B)$, and $u:\R^N\to\R$ given by $u(x):=m_s\int_{\partial B}M_s(x,z)g(z)\ dz.$  If $s\in(0,1)$, then $u\in C^{\infty}(B)\cap \cL^1_\sigma$, and if $s>1$, then $u\in C^{\infty}(B)\cap C_0^{s-1}(\overline{B})\cap \cL^1_\sigma$. Furthermore, $u$ is $s$-harmonic in $B$, that is, $(-\Delta)^s u(x)=(-\Delta)^m(-\Delta)^\sigma u(x)=0$ for $x\in B$.
\end{prop}
\begin{proof}
Let $u$ be as in the statement and $s>0$. By \cite[Proposition 1.5 and proof of Lemma 6.12]{AJS16}, $u\in C^{\infty}(B)\cap \cL^1_\sigma$ is $s$-harmonic in the sense of distributions, that is, 
$\int_{\R^N} u(x)(-\Delta)^s\phi(x)\ dx=0$ for all $\phi\in C^{\infty}_c(B).$  Then, by Lemma \ref{ibyp} and the fundamental lemma of calculus of variations, $u$ is $s$-harmonic also pointwisely in $B$. Finally, if $s>1$, we have that $u\in C_0^{s-1}(\overline{B})$, since ${\delta}^{1-s}u=\ch g\in C(\overline{B})$ (see \cite[Theorem 2.6]{GT}) and $M_s(x,z)=0$ for $x\in\R^{N}\backslash\overline{B}$.
\end{proof}

\subsection{Regularity}\label{reg:sec}

We begin with an elementary characterization of some functions in $\cH_0^s(B)$.
\begin{lemma}\label{l:ws}
 Let $m\in\N_0$, $\sigma\in(0,1)$, $s=m+\sigma$, and $\phi\in C^{m+1{,0}}(\overline{B})$, then ${\delta^s\varphi\in}\cH_0^s(B)$.
\end{lemma}
\begin{proof}
Let $s>0$ and $\varphi\in C^{m+1{,0}}(\overline{B})$. By  \cite[Section 4.3.2, equation (7)]{T78} (see also \cite[Remark 6.8]{AJS16}), it suffices to show that $u:={\delta^s}\varphi\in H^s(B)$, \emph{i.e.},
\begin{align*}
\|u\|_{H^s(B)}=\sum_{|\alpha|\leq m}\|\partial^\alpha u\|_{L^2(B)}+
\sum_{|\alpha|=m} \int_B\int_B\frac{|\partial^\alpha u(x)-\partial^\alpha u(y)|^2}{|x-y|^{N+2\sigma}}\ dydx<\infty.
\end{align*}
Observe that, by Leibniz rule, for every multi-index $\alpha\in\N^N$ and $k\in\{0,\ldots,m\}$ such that $|\alpha|=k$, the function $\partial^\alpha u$ is a sum of terms of the form 
$k_{\gamma,\beta} \partial^\gamma {\delta^s}\, \partial^\beta \phi$ for some multi-indices $\beta,\gamma\in\N^N$ such that $|\gamma|+|\beta|=k$ and some constant $k_{\gamma,\beta}\in \R$. Moreover, note that 
$(\sum_{i=1}^M a_i)^2 \leq 2^{M-1} \sum_{i=1}^M a_i^2$ for any $M\in\N$ and $\{a_i\}_{i=1}^M\subset\R$. Therefore, $u\in W^{m,\infty}(B)$ and it suffices to show that
\begin{align*}
\int_B\int_B\frac{|\partial^\gamma {\delta^s}(x)\, \partial^\beta \phi(x)-\partial^\gamma {\delta^s}(y)\, \partial^\beta \phi (y)|^2}{|x-y|^{N+2\sigma}}\ dydx<\infty,\qquad \text{ for all }\beta,\gamma\in\N^N, \ |\beta|+|\gamma|=m.
\end{align*}
However, this follows from the fact that $\varphi\in C^{m+1{,0}}(\overline{B})\subset H^{m+1}(B)\subset H^s(B)$, ${\delta^s}\in H^s(B)$ (see \cite[Remark 6.10]{AJS16} or \cite[Corollary 9]{DG2016}), and
\begin{align*}
|\partial^\gamma {\delta^s}(x)\, \partial^\beta \phi(x)-\partial^\gamma {\delta^s}(y)\, \partial^\beta \phi (y)|{ ^2}
\leq K(|\partial^\gamma {\delta^s}(x)-\partial^\gamma {\delta^s}(y)|^2+|\partial^\beta \phi(x)-\partial^\beta \phi (y)|^2 )
\end{align*}
where $K:=2\max\{\|\partial^\beta \phi\|_{L^\infty(B)}^2,\|\partial^\gamma {\delta^s}\|_{L^\infty(B)}^2\}$ for all $\beta,\gamma\in\N^N$ with $|\beta|+|\gamma|=m$, and the claim follows.
\end{proof}

\begin{lemma}\label{u:2:l}
Let $m\in\mathbb N_0$, $\sigma,\alpha\in(0,1)$, $s=m+\sigma$, $u:\R^N\to \R$ such that $\delta^{1-\sigma}u\in C^{m+\alpha}(\overline{B})$, $\dn^{k+\sigma-1}u=0$ at $\partial B$ for $k=0,\ldots,m$, $u=0$ on $\R^N\backslash\overline{B}$, and $(-\Delta)^su\in C^{1-\sigma+\alpha}(\overline{B})$. Then $u\in \cH_0^s(B)$ and, if $(-\Delta)^s u=0$ in $B$, then $u\equiv 0$ in $\R^N$.
\end{lemma}
\begin{proof}
Since $\dn^{k+\sigma-1}u=0$ for $k=0,\ldots,m$ and $\delta^{1-\sigma}u\in C^{m+\alpha}(\overline{B})$, we have that 
\begin{align*}
 0=\dn^{s-1}u(z)=
  \lim_{x\to z}\delta^{1-s}(x)u(x)\qquad \text{ for }\ z\in\partial B
\end{align*}
by Lemma \ref{prop:de}; therefore, $\delta^{1-s}u\in {C^{\alpha}(\overline{B})}$ and $u\in C^{{s-1}{+\alpha}}(\R^N)$ (recall that $u= 0$ on $\R^N\backslash\overline{B}$).  Now, assume that ${(-\Delta)^su}\in C^{1{-\sigma}{+\alpha}}(\overline{B})$, then \cite[Theorem 4, equation (6)]{G15:2} (using that $u\in \dot{C}^{s-1{+\alpha}}(\overline{B})=\{v\in C^{s-1+\alpha}(\R^N)\;:\; \supp\ v\subset \overline{B}\}$) yields that $\delta^{-s}u\in C^{{m+1}{+\alpha}}(\overline{B})$, and the claim $u\in \cH^s_0(B)$ follows by Lemma \ref{l:ws}. Finally, if $(-\Delta)^su=0$ in $B$, then $u\equiv 0$ in $\R^N$, by the uniqueness of weak solutions in $\cH_0^s(B)$ \cite[Corollary 3.6]{AJS16a}.
\end{proof}

Recall that $\|v \|_{C^{m+\sigma}(B)}$ is given in \eqref{Hn}.

\begin{lemma}\label{dnG:l}
Let $m\in \N_0$, $\sigma\in(0,1)$, $r>1$, $u\in L^{\infty}(B_{r}\backslash\overline{B})\cap \cL^1_{\sigma}$ such that 
\begin{align*}
\|(-\Delta)^\sigma u \|_{C^{m-\sigma-1+\beta}(B)}<\infty\qquad \text{ with }\quad m-\sigma-1+\beta>0\quad \text{  for some }\beta>0, 
\end{align*}
and let $w:=u-\Gamma_\sigma u$. Then $\delta^{1-\sigma} w\in C^{m+\alpha}(\overline{B})$ for any $0<\alpha<\beta$ and 
\begin{align*}
\dn^{\sigma} w(z)=\gamma_{N,\sigma}\lim_{x\to z}\int_{\R^N\backslash\overline{B}}\frac{u(x)-u(y)}{(|y|^2-1)^{\sigma}|x-y|^N}\ dy\qquad \text{ for }z\in \partial B.
\end{align*}
\end{lemma}
\begin{proof}
Since $u\in L^{\infty}(B_{r}\backslash\overline{B})\cap \cL^1_{\sigma}$, we have that $\cH_\sigma u\in C^\infty(B)\cap L^\infty(B)$ and $w=0$ on $\R^N\backslash\overline{B}$, by \cite[Lemma 2.5]{FW16}. Moreover, since $(-\Delta)^\sigma w=(-\Delta)^\sigma u$ in $B$, we have that $(-\Delta)^\sigma w\in L^\infty (B)$ and thus $w\in C^\sigma(\R^N)$, by \cite[Theorem 1.1]{RS12}.  Therefore, $\dn^{\sigma-1} w=0$ at $\partial B$ and, by \cite[equation (7.17)]{G15:2} (using $a=\sigma$ and $\varphi=\dn^{\sigma-1} w=0$), it follows that $\delta^{1-\sigma} w\in C^{m+\alpha}(\overline{B})$ for some  $0<\alpha<\beta$.  In particular, since $\dn^{\sigma-1} w =0$ on $\partial B$, 
 \begin{align*}
 \dn^{\sigma}w(z)=\lim_{x\to z} \delta^{-\sigma}(x)(u(x)-\cH_\sigma u(x))=\gamma_{N,\sigma}\lim_{x\to z}\int_{\R^N\backslash {B}}\frac{u(x)-u(y)}{(|y|^2-1)^{\sigma}|x-y|^N}\ dy\qquad \text{ for }z\in\partial B,
 \end{align*}
 by Lemma \ref{prop:de}.
\end{proof}

\section{The Poisson kernel}\label{sectionproof1}

In this section we show our main result concerning the Poisson kernel $\Gamma_s$ which is based on an induction argument.  We begin with a recurrence formula for $\Gamma_s$. Recall the constants $m_s$, $\gamma_{N,\sigma}$, and $k_{N,s}$ given in \eqref{c}.

\begin{prop}\label{Gamma:it} Let $m\in\N$, $\sigma\in(0,1)$, $s=m+\sigma$, and $\Gamma_s$ as in \eqref{PK}. Then 
\begin{align*}
\Gamma_{s}(x,y)=\Gamma_{s-1}(x,y)- m_s\int_{\partial B}M_s(x,z)\,\dn^{s-1}(\Gamma_{s-1}(\cdot,y))(z) \;dz, \qquad x\in B,\ y\in\R^N\setminus\overline{B},
\end{align*}
where 
\begin{align*}
\dn^{s-1}(\Gamma_{s-1}(\cdot,y))(z)=\frac{(-1)^{m-1}\gamma_{N,\sigma}}{(|y|^2-1)^{s-1}|y-z|^{N}},\qquad z\in\partial B,\ y\in \R^N\setminus\overline{B}.
\end{align*}
\end{prop}
\begin{proof}
Fix $y\in \R^N\backslash \overline{B}$ and let $p:B\to\R$ be given by $p(x):=|x-y|^{-N}\Big(1+\frac{1-|x|^2}{|y|^2-1}\Big).$  Note that
\begin{align}
(-1)^{m-1}\gamma_{N,\sigma}\left(\frac{1-|x|^2}{|y|^2-1}\right)^{s-1}p(x)
& =\Gamma_{s-1}(x,y)-\Gamma_{s}(x,y)\qquad \text{ for }x\in B.\label{poi1}
\end{align}
Moreover, a direct calculation shows that $-\Delta p(x) = 0$ for $x\in B$, and thus, by uniqueness, we can represent $p$ using the Poisson kernel for the Laplacian, namely,
\begin{align*}
p(x)=\int_{\partial B}\frac{2M_1(x,z)}{|y-z|^{N}}\ dz=2k_{N,1}\int_{\partial B}\frac{1-|x|^2}{|x-z|^N|y-z|^{N}}\ dz,\qquad x\in B.
\end{align*}
Therefore,
\begin{align}\label{poi2}
(-1)^{m-1}\gamma_{N,\sigma}\left(\frac{1-|x|^2}{|y|^2-1}\right)^{s-1}p(x)=(-1)^{m-1}\gamma_{N,\sigma}\frac{2s\, k_{N,1}}{k_{N,s}}\int_{\partial B} \frac{M_s(x,z)}{(|y|^2-1)^{s-1}|y-z|^{N}}\ dz.
\end{align}
The Proposition now follows from \eqref{poi1}, \eqref{poi2}, since
\begin{align*}
\frac{(-1)^{m-1}\gamma_{N,\sigma}}{(|y|^2-1)^{s-1}|y-z|^{N}}
=\lim_{B\ni\theta\to z} \frac{\Gamma_{s-1}(\theta,y)}{(1-|\theta|^2)^{s-1}}
=\dn^{s-1}(\Gamma_{s-1}(\cdot,y))(z),
\end{align*}
by Lemma \ref{prop:de} and the fact that $\dn^{\sigma+k}\Gamma_{s-1}(z,y)=0$ for $k\in\{-1,0,\ldots,s-2\}$.
\end{proof}

\begin{proof}[Proof of Theorem~\ref{poisson:thm}]
Let $\psi\in \cL_\sigma^1$ with $\psi=0$ in $B_{r}$ for some $r>1$ and fix $\rho\in(1,r)$. 
We show first that $u$ given by \eqref{u:def} belongs to  $C^{\infty}(B)\cap C^s(B_{r})\cap H^s(B_{\rho})$, $\rho\in(1,r),$ and that $u$ is a pointwise solution of \eqref{u:prob}.
Note that,
\begin{align}\label{u:dec}
u(x)= \delta(x)^s \phi(x), \quad \phi(x):=(-1)^m \gamma_{N,\sigma}\int_{\R^N\backslash\overline{B}_{r}} \frac{\psi(y)}{|x-y|^{N}(|y|^2-1)^s}\ dy, \qquad x\in\ B_{r}.
\end{align}
Thus $u\in C^{\infty}(B)\cap C^s(B_{r})$, and $u\in H^s(B_{\rho})$ by Lemma~\ref{l:ws} for $\rho\in(1,r)$.  We now argue by induction on $s$. For $s\in(0,1)$ the claim is known, see for example \cite[Lemma 2.5]{FW16} (or \cite[Lemma 1.13]{L72}, \cite[Theorem 1.2]{nicola}, \cite[Theorem 2.10]{B16}). Consider now $s=m+\sigma>1$ with $\sigma\in(0,1)$ and note that, by Proposition~\ref{Gamma:it} and a standard interchange of integrals, we have that
\begin{align}
u(x)&=\int_{\R^N\setminus\overline{B}}[\Gamma_{s-1}(x,y)- m_s\int_{\partial B}M_s(x,z)\,\dn^{s-1}(\Gamma_{s-1}(\cdot,y))(z) \;dz]\psi(y)\ dy \nonumber\\&
= \int_{\R^N\setminus\overline{B}}\Gamma_{s-1}(x,y)\psi(y)\ dy 
-m_s\int_{\partial B}M_s(x,z) 
\int_{\R^N\backslash\overline{B}_{r}} \frac{(-1)^{m-1}\gamma_{N,\sigma}\psi(y)}{(|y|^2-1)^{s-1}|y-z|^{N}}  
\ dy\ dz.\label{eq:pth}
\end{align}
Observe that the first term in \eqref{eq:pth} is $(s-1)$-harmonic in $B$, by the induction hypothesis, and the second term is $s$-harmonic in $B$, by Proposition \ref{lem:sharm-intro}. Since $(s-1)$-harmonic functions are $s$-harmonic (\emph{cf.}  \cite[Remark 6.15.4]{AJS16}), we have that $(-\Delta)^m (-\Delta)^\sigma u(x) = 0$ for all $x\in B$.  

\medskip

To argue uniqueness in $C^s(\overline{B})\cap H^s(B)$, let $v\in C^s(\overline{B})\cap H^s(B)$ be a pointwise solution of $(-\Delta)^s v=0$ in $B$ with $v=\psi$ in $\R^N\backslash\overline{B}$. Then $w:=u-v$ belongs to $\cH_0^s(B)$ (see \emph{e.g.} \cite[Remark 6.8]{AJS16} or \cite[Remark 3.6]{AJS16b}). Using that $C^\infty_c(B)$ is dense in $\cH_0^s(B)$ and integration by parts (see Lemma \ref{ibyp}), we have that $w$ is a weak solution (see Subsection \ref{Notation}) of $(-\Delta)^s w = 0$ in $B$ with $w=0$ in $\R^N\backslash\overline{B}$. Then, by uniqueness of weak solutions \cite[Corollary 3.6]{AJS16a}, we obtain that $w\equiv 0$ in $\R^N$, as desired.

\medskip

It remains to show \eqref{GsDsG}. Fix $s>1$, $s\notin\N$. By \eqref{Gs:decomp}, the function $x\mapsto \int_BG_s(x,y)\ dy$ belongs to $C^s_0(B)$. Let $x\in B$ and $y\in\R^N\setminus\overline{B}$, then
\begin{align}\label{DGs}
&\widetilde\Gamma_s(x,y):=-(-\Delta)_y^m(-\Delta)_y^\sigma G_s(x,y)
=(-\Delta)_y^m\int_B\frac{G_s(x,z)}{|z-y|^{N+2\sigma}}\;dz
=C\int_B\frac{G_s(x,z)}{|z-y|^{N+2s}}\;dz
\end{align}
for some constant $C\in \R$.  In particular, $\widetilde\Gamma_s(\cdot,y)\in C^\infty(B)\cap C^s_0(B)$ for $y\in\R^N\setminus\overline B$. Now, let $\psi$ be as above and $u:\R^N\to\R$ be given by
$u(x):=\int_{\R^N\setminus\overline{B}}\widetilde\Gamma_s(x,y)\psi(y)\;dy+\psi(x).$ We claim that $u$ is $s$-harmonic in $B$. Indeed, let $\phi\in C^\infty_c(B)$, then
\begin{align*}
\int_B u(x)\,&(-\Delta)^s\phi(x)\;dx=\int_B\int_{\R^N\setminus{B}}\widetilde\Gamma_s(x,y)\psi(y)\;dy\ (-\Delta)^s\phi(x)\;dx\\
&= \int_{\R^N\setminus{B}}\int_B\widetilde\Gamma_s(x,y)(-\Delta)^s\phi(x)\;dx\ \psi(y)\;dy\nonumber\\
&=-\int_{\R^N\setminus{B}}(-\Delta)^s_y \int_B G_s(x,y)(-\Delta)^s\phi(x)\;dx\ \psi(y)\;dy=-\int_{\R^N\setminus{B}}(-\Delta)^s\phi(y)\psi(y)\;dy,
\end{align*}
where the interchange of integral and derivatives is due to the dominated convergence theorem.  Thus, we have that
\begin{align*}
\int_{\R^N} u\,(-\Delta)^s\phi\;dx=\int_B u\,&(-\Delta)^s\phi\;dx+\int_{\R^N\setminus{B}}\psi(-\Delta)^s\phi\;dy=0\qquad \text{ for all }\phi\in C^\infty_c(B).
\end{align*}
Furthermore, by \eqref{DGs},
\begin{align*}
u(x)=\int_{\R^N\setminus\overline{B}}\widetilde\Gamma_s(x,y)\psi(y)\;dy
=C\int_BG_s(x,z)\int_{\R^N\setminus\overline{B}}\frac{\psi(y)}{|z-y|^{N+2s}}\;dy\;dz\qquad \text{ for }x\in B
\end{align*}
and therefore $u\in C^\infty(B)\cap H^s(B)\cap C^s(\overline{B})$, by the properties of the Green function $G_s$ (see \cite[Theorem 1.4]{AJS16} or \cite[Theorem 1.1]{AJS16b}). By Lemma \ref{ibyp}, $u$ is a pointwise solution of $(-\Delta)^s u=0$ in $B$ with $u=\psi$ in $\R^N\backslash\overline{B}$, and by the uniqueness shown in the first part of this proof we have that
\begin{align*}
u(x)=\int_{\R^N\setminus\overline{B}}\widetilde\Gamma_s(x,y)\psi(y)\;dy
=\int_{\R^N\setminus\overline{B}}\Gamma_s(x,y)\psi(y)\;dy\qquad \text{ for }x\in B.
\end{align*}
Since we may choose $\psi\in C(\R^N)$ with $\operatorname{supp}\psi\subset \R^N\setminus\overline{B}$ arbitrarily, the fundamental lemma of calculus of variations implies that $\widetilde\Gamma_s(x,y)=\Gamma_s(x,y)$ for all $x\in B$ and $y\in\R^N\backslash\overline{B}$.  Since $G_s(x,\cdot)\equiv 0$ if $x\not\in B$, we have that \eqref{GsDsG} also holds in $\R^N\backslash B$, and this ends the proof.
\end{proof}

 If the outside data $\psi$ does not belong to $\cL_\sigma^1$, then, in general, it is \emph{not} possible to compute pointwisely the fractional Laplacian $(-\Delta)^s u = (-\Delta)^m(-\Delta)^\sigma u$, with $u=\cH_s\psi$, see \eqref{fractional} and Remark \ref{pointwise:rmk} below.  However, a problem may still have a unique smooth \emph{distributional solution}.  To be more precise, for $s>0$ let 
\begin{align*}
 \cL^1_{s}:=\left\{u\in L^1_{loc}(\R^N)\;:\; \|u\|_{\cL^1_{s}}<\infty \right\}, \qquad \|u\|_{\cL^1_{s}}:=\int_{\R^N}\frac{|u(x)|}{1+|x|^{N+2s}}\ dx.
\end{align*}
We use the following estimate.
\begin{lemma}[See Lemma 3.9, \cite{AJS16}]\label{decay-s-smooth}
For any $s>0$, $\phi\in C^{\infty}_c(B)$ there is $C=C(\phi,s,N)>0$ such that
\begin{equation*}
|(-\Delta)^s\phi(x)|\leq \frac{C}{1+|x|^{N+2s}}\qquad \text{ for }x\in \R^N.
\end{equation*}
In particular, $u(-\Delta)^s\phi\in L^1(\R^N)$ for $u\in \cL^1_s$ and $\phi\in C^{\infty}_c(B)$. 
\end{lemma}
 
 Our result on (distributional) $s$-harmonic functions in $\cL^1_s$ is the following.
  \begin{cor}\label{Ls}
  Let $\sigma\in(0,1)$, $m\in\N$, $s=m+\sigma>1$, and $\psi\in \cL^1_{s}$ with $\psi=0$ in $B_r$, $r>1$. Then $u=\ch_s\psi\in \cL^1_s\cap C^{\infty}(B)\cap C^{s}(B_{1+r})\cap H^s(B_{\rho})$ for any $\rho\in(1,r)$. Moreover, $u$ satisfies \eqref{u:prob} in distributional sense, i.e.,
 \begin{align}\label{dist:sol}
  \int_{\R^N}u(x)(-\Delta)^s\phi(x)\ dx=0 \quad \text{ for all $\phi\in C^{\infty}_c(B)$}\qquad \text{ and }\qquad u=\psi\quad \text{ in }\quad \R^N\backslash B.
 \end{align}
 Furthermore, $u$ is the unique function in $C^s(\overline{B})\cap H^s(B)$ satisfying \eqref{dist:sol}.
 \end{cor}
 \begin{proof}
 Note that $u$ satisfies \eqref{u:dec} and hence by Lemma \ref{l:ws} it follows that 
 \begin{align}\label{reg:u:s}
 u:=\cH_s \psi\in \cL^1_s\cap C^{\infty}(B)\cap C^s(B_r)\cap H^s(B_\rho) \qquad\text{ for } \rho\in(1,r). 
 \end{align}
Let $\psi_n:=\psi \chi_{B_n}\in \cL^1_{\sigma}$ and $u_n:=\ch_s \psi_n$ for $n\in\N$.  Clearly $u_n\to u$ pointwisely in $\R^N$.  By Theorem \ref{poisson:thm}, $u_n\in C^{\infty}(B)\cap C^s(B_r)\cap H^s(B_\rho)\cap \cL^1_{\sigma}$, $\rho\in(1,r)$, and $u_n$ solves
\begin{align}\label{soln}
(-\Delta)^su_n=0\quad \text{ pointwise in $B$},\qquad \text{ and }\qquad u_n=\psi_n=\psi \chi_{B_n}\quad \text{ on $\R^N\backslash\overline{B}$.}  
\end{align}
Fix $\phi\in C^{\infty}_c(B)$, then by Lemma \ref{decay-s-smooth}, there is $C>0$ such that
\begin{align*}
 |u_n(x)(-\Delta)^s\phi(x)|\leq \frac{C}{1+|x|^{N+2s}}\Big( \int_{\R^N\backslash\overline{B}_r}\Gamma_s(x,y)|\psi(y)|\;dy+|\psi(x)|\Big)\qquad \text{ for }x\in\R^N,\ n\in\N.
\end{align*}
Then, by \eqref{reg:u:s}, \eqref{soln}, Lebesgue dominated convergence theorem, and Lemma \ref{ibyp},
\begin{align*}
\int_{\R^N}u\, (-\Delta)^s\phi\ dx=\lim_{n\to\infty} \int_{\R^N}u_n\, (-\Delta)^s\phi\ dx=\lim_{n\to\infty} \int_{B}(-\Delta)^m(-\Delta)^\sigma u_n\, \phi\ dx=0.
\end{align*}
The uniqueness in $C^s(\overline{B})\cap H^s(B)$ now follows similarly as in the proof of Theorem~\ref{poisson:thm}. 
 \end{proof}

\subsection{On maximum principles}

In this subsection, we show Theorem \ref{even:thm2} and Theorem \ref{connected:cor}. In particular, we use the explicit formula for the Poisson kernel $\Gamma_s$ to yield information about the sign of the Green function associated to two disjoint balls. This has a close relationship with the validity or failure of maximum principles for $(-\Delta)^s$.  For the reader's convenience, we provide in the appendix (see Proposition \ref{existenceuniqueness}) a proof of the existence of Green functions $G_\Omega$ for $(-\Delta)^s$ in smooth bounded domains $\Omega$.  In particular, $G_B\equiv G_s$ with $G_s$ as in \eqref{green}.

\begin{prop}\label{disjoint-sets-prop}
Let $N\in\mathbb N$, $s>0$, and $A\subset \R^N$ be an open smooth set such that $B_{r}\cap A=\emptyset$ for some $r>1$. The unique Green function $G_{\Omega}$ of $(-\Delta)^s$ in $\Omega:=B\cup A$ given by Proposition \ref{existenceuniqueness} satisfies
\begin{align}
G_{\Omega}(x,y) &= G_B(x,y)+\int_{A}\Gamma_s({x},z)\,G_{\Omega}({y},z)\;dz&&\text{ for }\  x,y\in B,\ x\neq y, \label{sameball}\\
G_{\Omega}(x,y) &= \ch_s[G_{\Omega}(\cdot,{y})](x)=\int_{A}\Gamma_{s}({x},z)\,G_{\Omega}(z,{y})\;dz &&\text{ for }\   x\in B,\ y\in A\label{diffball}.
\end{align}
\end{prop}
\begin{proof}
To show \eqref{sameball}, let $\psi\in C^\infty_c(B)$ and $v:\R^N\to\R$ be given by
\begin{align*}
v(x) &:= \int_\Omega G_\Omega(x,y)\psi(y)\;dy-\int_B G_B(x,y)\psi(y)\;dy - \ch_s \Big[\int_B G_\Omega(\cdot,y)\,\psi(y)\;dy \Big](x),\quad x\in B,
\end{align*}
and $v=0$ on $\R^N\backslash\overline{B}$.  Note that $v\in \cH^s_0(B)$ and solves weakly $(-\Delta)^s v = 0$ in $B$.  Then $v \equiv 0$ in $\R^N$, by uniqueness (\cite[Corollary 3.6]{AJS16a}). Since $\psi$ is arbitrary, \eqref{sameball} follows from the fundamental lemma of calculus of variations. Equality \eqref{diffball} follows similarly: let $\phi\in C^\infty_c(A)$ and $w\in \cH^s_0(B)$ given by
$w(x) := \int_\Omega G_\Omega(x,y)\phi(y)\;dy - \ch_s[\,\int_{A} G_\Omega(\cdot,y)\,\phi(y)\;dy](x)$. Then $w$ solves weakly $(-\Delta)^s w=0$ in $B$ and, as before, $w=0$ in $B$, by uniqueness. Since $\phi$ is arbitrary, \eqref{diffball} follows as before. 
\end{proof}

Let us now introduce some notation that we use below in the proof of Theorem \ref{even:thm2}: for $t>2$, set $B^t=B_1(te_1)$, $\Omega(t)=B\cup B^t$, and let us analyze the corresponding Green function $G_{\Omega(t)}$, $t>2$.  Let $R_t\ :\ \R^N \to\ \R^N,$ $x=(x_1,x') \mapsto\ (t-x_1,x'),$ denote the reflection with respect to the hyperplane $\{2x_1=t\}$. We note that $B^t=R_t B$ and $R_t\Omega(t)=\Omega(t)$.   
\begin{lemma}\label{Rt:lem} 
Let $x\in B$ and $t\geq2$, then $1-|x|^2\leq |R_tx|^2-1$.
\end{lemma}
\begin{proof}
Let $x\in B$ and $t\geq2$.  Since $x_1<1$ we have
\[
|R_tx|^2+|x|^2=t^2-2tx_1+2x_1^2+2|x'|^2\geq 2t(1-x_1)+2x_1^2\geq 4(1-x_1)+2x_1^2\geq 2,
\]
and the claim follows.
\end{proof}

For $t\geq2$, $y\in B^t$, and $z\in B\subset\R^N\backslash B^t$, let
\begin{align}\label{reflected-pois}
\Gamma_{B^t}(y,z)\ :=\ \Gamma_{B}(y-te_1,z-te_1)\ =\ \Gamma_B(R_t y, R_t z), 
\end{align}
where $\Gamma_B:=\Gamma_s$ is given in \eqref{PK}.  Then, applying Proposition \ref{disjoint-sets-prop} to the domain $\Omega=\Omega(t)$, we obtain
\begin{align}
G_{\Omega}(x,y) &= G_B(x,y)+\int_{B^t}\Gamma_B({x},z)\,G_{\Omega}({y},z)\;dz&&\text{ for }\  x,y\in B,\ x\neq y, \label{sameball2}\\
G_{\Omega}(x,y) &= \int_{B^t}\Gamma_B({x},z)\,G_{\Omega}(z,{y})\;dz &&\text{ for }\   x\in B,\ y\in B^t
\nonumber
\end{align}
Since $\Omega$ consists of two balls, using equation \eqref{reflected-pois} we can also write 
\begin{equation}
G_{\Omega}(x,y) = \int_B\Gamma_{B^t}(y,z)\,G_{\Omega}(x,z)\;dz 
=\int_B\Gamma_B(R_ty,R_tz)\,G_{\Omega}(x,z)\;dz\qquad\text{ for }\   x\in B,\ y\in B^t\label{diffball3}.
\end{equation}

\begin{proof}[Proof of Theorem \ref{even:thm2}]  Fix $t>T_0$, with $T_0$ as in \eqref{T0:def}. To simplify the notation, we write $\Omega$ instead of $\Omega(t)$. Fix $x\in B$, $y\in B^t$, and $m$ even. 
Substituting \eqref{sameball2} in the last term of \eqref{diffball3}, using the positivity of $G_B$ we deduce 
\begin{align}\label{meven}
G_\Omega(x,y)&\geq \int_{B}\Gamma_{B}(R_ty,R_tz)\left(\int_{B^t}\Gamma_{B}(z,w)\,G_\Omega(x,w)\;dw\right)dz{,\qquad x\in B,\ y\in B^t}.
\end{align}
Then, an interchange of integrals yields that
\begin{equation*}
G_\Omega(x,y) \geq \int_{B^t}G_\Omega(x,w)K(y,w) \ dw{,\qquad x\in B,\ y\in B^t}.
\end{equation*}
where
\[
K(a,b)\ :=\ \int_{B}\Gamma_{B}(R_ta,R_tz)\,\Gamma_{B}(z,b)\;dz\geq 0\qquad \text{ for } \ a,b\in B^t,
\]
by \eqref{PK}.  Iterating this procedure we have that 
\begin{equation*}
\begin{split}
G_\Omega(x,y) &\geq\ \int_{B^t}\left(\int_{B^t}G_\Omega(x,x^2)\,K(w,x^2)\;dx^2\right)K(y,w)\;dw \\
&= \int_{B^t}G_\Omega(x,x^2)\left(\int_{B^t}K(y,x^1)\,K(x^1,x^2)\;dx^1\right)dx^2
\end{split}
\end{equation*}
and, after $n$ iterations,
\begin{equation}\label{n}
G_\Omega(x,y)\geq \int_{B^t}G_\Omega(x,x^n)\left(\int_{B^t}\ldots\int_{B^t}K(y,x^1)\cdots K(x^{n-1},x^n)\;dx^1\ldots dx^{n-1}\right)dx^n.
\end{equation}
We now estimate $K$. It follows from Lemma \ref{Rt:lem} and the definition of $K$ that, for $a,b\in B^t$,
\begin{align*}
K(a,b)&=\gamma_{N,\sigma}^2\left(\frac{1-|R_ta|^2}{|b|^2-1}\right)^s\int_B\frac{1}{|R_ta-R_tz|^N|b-z|^N}
\left(\frac{1-|z|^2}{|R_tz|^2-1}\right)^sdz \\
&\leq \gamma_{N,\sigma}^2\left(\frac{|a|^2-1}{|b|^2-1}\right)^s\int_B\frac{1}{|a-z|^N|b-z|^N}\;dz\leq \left(\frac{|a|^2-1}{|b|^2-1}\right)^s\frac{|B|\gamma_{N,\sigma}^2}{(t-2)^{2N}}.
\end{align*}
	
Therefore,
\[
\int_{B^t}\ldots\int_{B^t}K(y,x^1)K(x^1,x^2)\cdots K(x^{n-1},x^n)\;dx^1\ldots dx^{n-1}\leq
\Big(\frac{|B|\gamma_{N,\sigma}}{(t-2)^{N}}\Big)^{2(n-1)}\left(\frac{|y|^2-1}{|x^n|^2-1}\right)^s,
\]
and we deduce from \eqref{n} that
\begin{align*}
G_\Omega(x,y)
&\geq -\Big(\frac{|B|\gamma_{N,\sigma}}{(t-2)^{N}}\Big)^{2(n-1)}\int_{B^t}|G_\Omega(x,z)|\left(\frac{|y|^2-1}{|z|^2-1}\right)^sdz\\
&\geq -\Big(\frac{|B|\gamma_{N,\sigma}}{(t-2)^{N}}\Big)^{2(n-1)}\left(\frac{|y|^2-1}{t(t-2)}\right)^s\int_{B^t}|G_\Omega(x,z)|dz\to 0\qquad \text{ as }n \to\infty,
\end{align*}
because $|B|\gamma_{N,\sigma}(t-2)^{-N}<1$, by \eqref{T0:def} and because $t>T_0$.  Since $(x,y)\in B\times B^t$ was taken arbitrarily we have $G_\Omega\geq 0$ in $B\times B^t$. Then \eqref{niBe} follows by symmetry and, by \eqref{sameball}, we also obtain that \eqref{inB} holds if $m$ is even.
	
On the other hand, if $m$ is odd, then instead of \eqref{meven} we have
\begin{align*}
G_\Omega(x,y)&\leq \int_{B}\Gamma_{B}(R_ty,R_tz)\left(\int_{B^t}\Gamma_{B}(z,w)\,G_\Omega(x,w)\;dw\right)dz,
\end{align*}
and \eqref{niBo}, \eqref{inB} follow from a similar reasoning as in the case $m$ even. This ends the proof.
\end{proof}

We conclude this part with the proof of Theorem \ref{connected:cor}.  For this, we recall first a result from \cite{AJS16a}.

\begin{thm}[Theorem 1.1 in \cite{AJS16a}]\label{mp-fails}
Let $N\in\N$, $D\subset\R^N$ be an open set, $s\in(k,k+1)$ for some $k\in\mathbb N$ odd, $A$ be a nonempty ball compactly contained in $\R^N\setminus D$, and $\Omega=D\cup A$. There is a smooth positive function $f\in C^\infty(\overline{\Omega})$ such that the problem $(-\Delta)^{s}u =f$ in $\Omega$, admits a sign-changing weak solution $u\in\cH_0^{s}(\Omega)\cap C(\R^N)\cap C^{\infty}(\Omega)$ with $u\lneq 0$ in $D$ and $u\gneq 0$ in $A$.
\end{thm}

\begin{proof}[Proof of Theorem \ref{connected:cor}]
Let $\Omega_n$ as in the statement, $f\in C^\infty(\overline{\Omega})$, and $u\in\cH_0^{s}(\Omega)\cap C(\R^N)$ be the functions given by Theorem \ref{mp-fails} for $\Omega=D\cup A$ with $D=B_1(0)$ and $A=B_1(3e_1)$. For $n\in\N$ let $f_{n}\in L^\infty(\Omega_{n})$ be given by $f_{n}:=f \chi_\Omega \gneq 0$, where $\chi_\Omega$ is the characteristic function of the set $\Omega$. Let $u_{n}\in\cH_0^{s}(\Omega_{n})$ be the weak solution of $(-\Delta)^{s}u_{n}=f_{n}$ in $\Omega_{n}$ given by \cite[Corollary 3.6]{AJS16a}, \emph{i.e.}, 
\begin{align}\label{weak:for}
\cE_s(u_{n},\varphi)=\int_{\Omega_{n}}f_{n}\varphi \ dx=\int_{\Omega}f \varphi \ dx\qquad \text{ for all }\varphi\in\cH_0^{s}(\Omega_{n}).  
\end{align}
For $v\in\cH_0^s(\Omega_{1})$, let $\|v\|:=\cE_s(v,v)^{\frac{1}{2}}$.  By the Poincar\'{e} inequality \cite[Proposition 3.3]{AJS16a} we have that $\|\cdot\|$ is a norm and $(\cH_0^s(\Omega_{1}),\cE_s(\cdot,\cdot))$ is a Hilbert space.  By testing \eqref{weak:for} with $u_{n}$ and using Cauchy-Schwarz inequality, Poincar\'{e} inequality (see \cite[Proposition 3.3]{AJS16a}), and $\cH_0^s(\Omega_{n})\subset \cH_0^s(\Omega_{1})$, we have that 
\begin{align*}
\|u_{n}\|^2=\cE_s(u_{n},u_{n}) \leq \|u_n\|_{L^2(\Omega_1)}\|f\|_{L^2(\Omega)}\leq \lambda^{-1}_{1,s}(\Omega_{1}) \|u_{n}\|\|f\|_{L^2(\Omega)}.
\end{align*}
Therefore $(u_{n})_{n\in\N}$ is uniformly bounded in $\cH_0^s(\Omega_{1})$ and then there is $u^*\in \cH_0^s(\Omega_{1})$ such that $u_{n}\rightharpoonup u^*$ weakly in $\cH_0^s(\Omega_{1})$ and $u_{n}\to u^*$ strongly in $L^2(\R^N)$ as $n\to \infty$. Since $\operatorname{supp} u_{n}\subset \Omega_{n}$ we have that $\operatorname{supp} u^*\subset \overline{\Omega}\cup L$ and we may assume without loss of generality that $u^*=0$ on $L$, because $L$ has measure zero (since $N\geq 2$). Thus $u^*\in \cH_0^{s}(\Omega)$ and for any $\varphi\in\cH_0^{s}(\Omega)\subset \cH_0^s(\Omega_{1})$,
\begin{align*}
\cE_s(u^*,\varphi)=\lim_{n\to\infty}\cE_s(u_{n},\varphi) =\int_{\Omega}f\varphi \ dx.  
\end{align*}
By the uniqueness of solutions given by \cite[Corollary 3.6]{AJS16a}, we have that $u^*=u$ a.e. in $\Omega$. The result now follows, since $u_{n}(x)\to u^*(x)=u(x)$ for a.e. $x\in\Omega$ and $u\in C(\R^N)$ satisfies that $u\lneq 0$ in $B_1(0)$ and $u\gneq 0$ in $B_1(3e_1)$, by Theorem \ref{mp-fails}. 
\end{proof}

\section{Representation formulas and explicit solutions with boundary kernels}\label{IRF:sec}

We begin this section with the study of $s$-harmonic functions which are zero in $\R^N\backslash\overline{B}$.
Recall from \eqref{dn-vs-partial} that $\delta(x):=(1-|x|^2)_+$ for $x\in\R^N$ and, for $k\in\N_{0}$, $\sigma\in(0,1]$
\begin{equation*}
\dn^{k+\sigma-1}u(z)=\frac{(-1)^k}{k!}\lim_{x\to z}\frac{\partial^k}{\partial (|x|^2)^k}[\delta(x)^{1-\sigma} u(x)]\ \  \text{ and }\ \ \dn^{k}u(z)=\frac{(-1)^k}{k!}\lim_{x\to z}\frac{\partial^k}{\partial (|x|^2)^k}u(x).
\end{equation*}

\begin{thm}\label{thm:sHTOD}
Let $m\in\N_0$, $\sigma\in(0,1]$, $s=m+\sigma$, $g_k\in C^{m-k{,0}}(\partial B)$ for $k=0,\ldots,m$. Let $u:\R^N\to\R$ be given by 
\begin{align}\label{sEK:u}
 u(x)=\sum_{k=0}^{m}\ \int_{\partial B}E_{k,s}(x,\theta)\ g_k(\theta)\ d\theta,
\end{align}
where $E_{k,s}$ are given by \eqref{sEK}. Then $u\in C^{\infty}(B)$ with $\delta^{1-\sigma}u\in C^{m{,0}}(\overline{B})$ is a solution of $(-\Delta)^s u=0$ in $B$ satisfying $u=0$ on $\mathbb R^N\backslash\overline{B}$ and
\begin{align}
\dn^{k+\sigma-1} u= g_k\quad \text{ on }\ \ \partial B\qquad \text{ for }\ \ k=0,1,\ldots,m. \label{sE:p}
\end{align}
Furthermore, if $t\in (0,\infty)\backslash \N$ and $g_k\in C^{m+t-k}(\partial B)$ for $k=0,\ldots,m$, then $\delta^{1-\sigma}u\in C^{m+t}(\overline{B})$.
\end{thm}
\begin{proof}
 Let $g_k\in C^{m-k{,0}}(\partial B)$ for $k=0,\ldots,m$ and $v$ as in \eqref{sEK:u} with $s=m+1$. By \cite[Satz 2 \& 3]{E75-2}, we have that $v\in C^{\infty}(B)\cap C^{m+1{,0}}(\overline{B})$ is the unique solution of
 \begin{align}\label{E:p}
(-\Delta)^{m+1} v=0\ \ \text{ in }B\quad \text{ and }\quad 
\dn^k v= g_k\ \ \text{ on }\ \ \partial B\quad \text{ for }\ \ k=0,1,\ldots,m.
\end{align}
Let $u:=\delta^{\sigma-1} v$. Then, by Lemma \ref{Almansi:l} applied to $v$, there are unique functions $h_k\in C(\partial B)$ for $k=0,\ldots,m$, such that 
 $u(x)=\sum_{k=0}^m\ \int_{\partial B} M_{k+\sigma}(x,\theta)\ h_k(\theta)\ d\theta$.  Therefore $u$ is $(m+\sigma)$-harmonic in $B$ by Proposition \ref{lem:sharm-intro}. Moreover, by \eqref{E:p},
 $\dn^{k+\sigma-1}u=\dn^k[\delta^{1-\sigma}u]=\dn^kv=g_k$ on $\partial B$ for $k=0,\ldots,m$, and the boundary conditions \eqref{sE:p} follow. Finally, let $t\in (0,\infty)\backslash \N$ and $g_k\in C^{m+t-k}(\partial B)$ for $k=0,\ldots,m$. By Schauder theory (see \cite[Theorem 2.19]{GGS10} or \cite[Theorem 9.3]{ADN59}), the problem \eqref{E:p} has a unique (mild) solution in $C^{m+t}(\overline{B})$, which then must be given by $v$, and therefore $\delta^{1-\sigma} u=v\in C^{m+t}(\overline{B})$.
\end{proof}

\begin{remark}\label{noconstants}
 In Edenhofer's original formulation \cite{E75-2} the kernels are 
\begin{align}\label{mc}
E_{k,m}(x,\theta)=\frac{(-1)^{m-1}}{\omega_N(m-1)!}(1-|x|^2)^m\int_{\partial B}\binom{m-1}{k}\varphi_k(\theta) 
\frac{\partial^{m-1-k}}{\partial(|y|^2)^{m-1-k}}\Big( \frac{|y|^{N-2}}{|x-y|^N} \Big)\Bigg|_{y=\theta}\ d\theta,
\end{align}
where $\varphi_k(\theta)=\lim\limits_{y\to\theta}\frac{\partial^{k}}{\partial(|y|^2)^{k}}u(y)$. Due to our normalization constants in the definition of $\dn^{t}$, formula \eqref{mc} is equivalent to \eqref{sEK}, using the prescribed functions $g_k=\dn^k u$.  We also remark that the assumption $g_k\in C^{m-k{,0}}(\partial B)$ from Theorem \ref{main:thm} is not technical, if $g_k$ are merely continuous functions, then the Dirichlet (polyharmonic) problem may not have a solution, see \cite{E75-1}.

\end{remark}

\begin{remark}\label{K:Equiv}
In some contexts, the use of normal derivatives $(-\partial_\nu)^k$ may be more natural than the use of $(\frac{\partial}{\partial |x|^2})^k$.  In Subsection \ref{remarks-derivative} we mentioned that these two notions of derivatives are different, but equivalent. To exemplify this, let 
\begin{align*}
T_{k+\sigma-1} u(z):=\lim_{x\to z}(-\partial_\nu)^k(\delta^{1-\sigma}(x) u(x))\qquad \text{ for }z\in\partial B.
\end{align*}
We recall that all limits are meant in the normal direction from inside $B$.  We show how Edenhofer kernels $E_{k,s}$ can be used to construct a function $u$ satisfying 
\begin{align*}
(-\Delta)^s u(x)=(-\Delta)^m (-\Delta)^\sigma u(x)=0\qquad \text{ for }x\in B
\end{align*}
with the boundary conditions
\begin{align}
T_{k+\sigma-1} u=0\quad \text{ and } \quad T_{(m-1)+\sigma-1} u(z)=\varphi\qquad \text{ at }\partial B \label{nu:der}
\end{align}
for $k\in\{0,\ldots,m\}\backslash \{m-1\}$ and for some given $\varphi\in C^{1{,0}}(\partial B)$.  By \eqref{nu:der} and some elementary calculations (recall \eqref{page10}), we have, for $z\in\partial B$, that
\begin{align*}
\dn^{k+\sigma-1} u(z)&=0\qquad \text{ for }k\leq m-3,\\
\dn^{(m-1)+\sigma-1} u(z)&=\frac{1}{(m-1)!} \lim_{x\to z}\frac{(-\partial_\nu)^{m-1}}{2^{m-1}} (\delta^{1-\sigma}(x) u(x)) =\frac{\varphi(z)}{2^{m-1}(m-1)!},\\
\dn^{m+\sigma-1} u(z)&=\frac{1}{m!} \lim_{x\to z}\sum_{j=0}^{m}\frac{j}{2^m }(-\partial_\nu)^{m-1} (\delta^{1-\sigma}(x) u(x))=\frac{m+1}{2^{m+1}(m-1)!}\varphi(z).
\end{align*}
Moreover, for $x\in B$ and $\theta\in\partial B$,
\begin{align}\label{expE}
E_{m,s}(x,\theta)&=\frac{1}{ \omega_N }\frac{\delta(x)^s}{|x-\theta|^N}\quad\text{ and }\quad E_{m-1,s}(x,\theta)=\frac{\delta(x)^s(N\delta(x)-(N-4)|x-\theta|^2)}{4\omega_N \ |x-\theta|^{N+2}}.
\end{align}
Therefore,
\begin{align*}
u(x)&=\int_{\partial B}E_{m-1,s}(x,\theta)\dn^{s-2} u(\theta)+E_{m,s}(x,\theta)\dn^{s-1} u(\theta)\ d\theta\\
&=\frac{1}{2^{m+1}(m-1)!\omega_N}\int_{\partial B} \frac{\delta(x)^{s}}{|x-\theta|^{N+2}}(N\delta(x)-(N-3-m)|x-\theta|^2)\varphi(\theta)\ d\theta.
\end{align*} 
If $\sigma=1$, this coincides with the Dirichlet $(m+1)$-harmonic case, see \cite[page 160]{GGS10}.  Similarly, the solution of $(-\Delta)^s u=0$ in $B$ with 
$T_{k+\sigma-1} u=0$ at $\partial B$ for $k\leq m-1$ and $T_{m+\sigma-1} u=\varphi$ at $\partial B$ for a given $\varphi\in C(\partial B)$ is
\begin{align*}
 u(x)&=\int_{\partial B} E_{m,s}(x,\theta)\dn^{s-1} u(\theta)\ d\theta=\frac{1}{ 2^{m}m!\omega_N}\int_{\partial B}\frac{\delta(x)^{s}}{|x-\theta|^N}\varphi(z)\ d\theta.
\end{align*}
\end{remark}

Our next result provides a representation formula for $s$-harmonic functions which are zero in $\R^N\backslash\overline{B}$.

\begin{thm}\label{u:thm}
Let $m\in\N_0$, $\sigma, \alpha\in(0,1 ]$, $s=m+\sigma$ and $u\in C^{2s+\alpha}(B)$ such that
\begin{align*}
\delta^{1-\sigma}u\in C^{ m+\alpha}(\overline{B}),\qquad  u=0\ \ \text{ on }\mathbb R^N\backslash\overline{B},\qquad \text{ and }\qquad (-\Delta)^s u=0\ \  \text{  in }B.
\end{align*}
Then $u$ satisfies \eqref{sEK:u} with $g_k:=\dn^{k+\sigma-1} u$ on $\partial B$ for $k=0,1,\ldots,m.$
\end{thm}
\begin{proof}
Let $u$ as in the statement, $v$ as in \eqref{sEK:u} with $g_k:=\dn^{k+\sigma-1} u$ on $\partial B$ for $k=0,1,\ldots,m$, and set $w:=u-v$.  By Theorem \ref{thm:sHTOD}, we have that $\delta^{1-\sigma}w\in C^{ m+\alpha}(\overline{B})$ and $\dn^{k+\sigma-1}w=0$ at $\partial B$ for $k=0,\ldots,m$, which yields that $w\equiv 0$ in $\R^N$, by Lemma \ref{u:2:l}.
\end{proof}

As a consequence of Theorem \ref{u:thm}, we can infer the following relationship between the harmonic extensions $\cH_s$ and $\cH_\sigma$.

\begin{cor}\label{I-2}
Let $m\in \N$, $\sigma\in(0,1)$, and $s=m+\sigma$. Moreover, let $g\in \cL^1_{\sigma}$ such that $g=0$ in $B_{r}$ for some $r>1$. Then
\begin{equation}\label{I}
\cH_s g(x)= \cH_\sigma g(x) -\sum_{k=0}^m \int_{\partial B} E_{k,s}(x,\theta)\dn^{k+\sigma-1}\cH_\sigma g(\theta)\ d\theta\qquad \text{ for }x\in B, 
\end{equation}
where $\cH_{\sigma}g, \cH_sg$ are as in \eqref{GHHs}.
\end{cor}
\begin{proof}
Let $g$ as in the statement and denote by $w:=\cH_s g-\cH_\sigma g$, then, using Lemma~\ref{prop:de},
 $\dn^{k+\sigma-1} \cH_s g= 0$ on $\partial B$ for $k=0,\ldots,m.$  Moreover, $\delta^{1-\sigma} w\in C^{\infty}(\overline{B})$, since
\begin{align*}
 \delta^{1-\sigma}(x) w(x)=\gamma_{N,\sigma}\int_{\R^N\backslash {B_{r}}} \Bigg((-1)^m\frac{\delta(x)^{m+1}}{(|y|^2-1)^s}-\frac{\delta(x)}{(|y|^2-1)^\sigma}\Bigg) \frac{g(y)}{|x-y|^N}\ dy\qquad \text{ for }x\in B
\end{align*}
and this integral has no singularity. Theorem \ref{u:thm} applied to $w$ yields \eqref{I}.
\end{proof}

We can now proceed to the proof of our main results.

\begin{proof}[Proof of Theorem \ref{main:thm}]
Let $\cG_s u$ and $\cH_s u$ as in \eqref{GHHs} and $f,h$ as in the statement. Observe that 
\begin{align*}
\dn^{k+\sigma-1} \int_B G_s(\cdot,y)f(y)\ dy=\dn^{k+\sigma-1}\int_{\R^N\backslash\overline{B}}\Gamma_s(\cdot,y)h(y)\ dy =0\quad \text{ on } \partial B\quad \text{ for }k=0,1,\ldots,m, 
\end{align*}
by Lemma \ref{prop:de} and the explicit formulas for $G_s$ and $\Gamma_s$, see \eqref{green} and \eqref{PK}. By \cite[Theorem 1.1]{AJS16b} (or \cite[Theorem 1.4]{AJS16}), $G_s$ is a Green function for $(-\Delta)^s$ in $B$, and therefore the result follows from Theorem \ref{thm:sHTOD} and Theorem \ref{poisson:thm}.  
\end{proof}

\begin{proof}[Proof of Theorem \ref{uniqueness:thm}]
Let $u$ as in the statement, $\cG_s u$, $\cH_s u$ as in \eqref{GHHs}, and $v$ given by \eqref{sEK:u} with $g_k:=\dn^{k+\sigma-1} u$ on $\partial B$ for $k=0,1,\ldots,m$. Observe that $\dn^{k+\sigma-1} \cG_s u=0$ and $\dn^{k+\sigma-1} \cH_s u=0$ on $\partial B$ for $k=0,1,\ldots,m$. Then, by Theorem \ref{poisson:thm}, \ref{u:thm}, and the fact that $G_s$ is a Green function for $(-\Delta)^s$ in $B$ (see \cite[Theorem 1.1]{AJS16b} or \cite[Theorem 1.4]{AJS16}), it follows that $v= u - \cG_s u -  \cH_s u$ in $B$, as claimed.
\end{proof}

\begin{proof}[Proof of Theorem \ref{uniqueness:thm:2}]
By Lemma \ref{dnG:l} we may apply Theorem \ref{uniqueness:thm} to $w=u-\cH_\sigma u$, and the result follows.
\end{proof}

\begin{proof}[Proof of Corollary \ref{Almansi:cor}]
Let $\alpha,\sigma\in(0,1]$, $m\in\N$, $s=m+\sigma$, and $u\in C^{2s+\alpha}(B)$ be such that $(-\Delta)^s u=0$ in $B$, $\delta^{1-\sigma}u\in C^{ m+\alpha}(\overline{B})$, and $u=0$ in $\R^N\backslash\overline{B}$.  Then, for $x\in B$,
\begin{align*}
u(x)=\sum_{k=0}^{m}\ \int_{\partial B}E_{k,s}(x,\theta)\ g_k(\theta)\ d\theta
=\delta^{1-\sigma}(x)\varphi(x),\quad \varphi(x):=\sum_{k=0}^{m}\ \int_{\partial B}E_{k,m+1}(x,\theta)\ h_k(\theta)\ d\theta,
\end{align*}
with $h_k:=\dn^{k+\sigma-1} u$ at $\partial B$ for $k=0,1,\ldots,m,$ by Theorem \ref{u:thm} and \eqref{sEK}; in particular, $u\in C^\infty(B)$, by Theorem \ref{thm:sHTOD}.  Since $\varphi$ is $(m+1)-$harmonic (by Theorem \ref{main:thm}) we have by Lemma \ref{Almansi:l} and \eqref{martin} that
\begin{align*}
u(x)=\delta(x)^{\sigma-1}\sum_{k=0}^{m}\ \int_{\partial B}M_{k+1}(x,\theta)\ g_k(\theta)\ d\theta
=\sum_{k=0}^{m}\ \int_{\partial B}M_{k+\sigma}(x,\theta)\ g_k(\theta)\ d\theta
\end{align*}
for some uniquely determined functions $g_k\in C(\partial B)$.  Finally, if $\psi:=\delta^t u$ for some $t>-\sigma$, then $\psi$ is $(s+t)$-harmonic in $B$, by Proposition \ref{lem:sharm-intro}.
\end{proof}

For the proof of Lemma \ref{explicit-edenhofer}, we recall the differential recurrence formula given in \cite[Lemma 3.1]{AJS16b} or \cite[Lemma 6.1]{AJS16} (note that in \cite{AJS16b,AJS16} differentiation is taken in $x$ rather than in $y$), i.e. we have for $s>1$
\begin{equation}\label{green-recurrence}
(-\Delta)_y G_{s}(x,y)=G_{s-1}(x,y)-4(s-1) k_{N,s}\,P_{s-1}(x,y),	\qquad \ x,y\in\R^N,\ x\neq y,
\end{equation}
where, for $x,y\in \R^N$, 
\begin{align*}
P_{s-1}(x,y)\ :=\ \frac{\delta^{s-1}(x)\delta^{s-2}(y)(1-|x|^2|y|^2)}{{[x,y]}^N},\qquad [x,y]\ :=(|x|^2|y|^2-2x\cdot y+1)^\frac{1}{2}.
\end{align*}

\begin{proof}[Proof of Lemma \ref{explicit-edenhofer}]
Fix $x\in B$ and $\theta\in\partial B$, and note that
\[
E_{m,s}(x,\theta)=\omega_N^{-1}\delta(x)^s|x-\theta|^{-N}= m_{s}M_s(x,\theta)= m_{s}\lim\limits_{B\ni y\to \theta} \frac{G_s(x,y)}{\delta^s(y)}=m_{s}\dn^{s}G_s(x,\theta).
\]
and \eqref{comp1a} follows, see \eqref{expE}.  Furthermore, if $m\geq 1$, then since $\dn^{s-2}G_{s-1}(x,\cdot)=0$ on $\partial B$ and 
\begin{align}\label{consta}
4(s-1)k_{N,s}=\frac{2k_{N,1}}{m_{s-1}}=(m_{s-1}\omega_N)^{-1},
\end{align}
a direct computation using \eqref{green-recurrence} yields
\begin{align*}
m_{s-1}&\dn^{s-2}[\Delta G_s(x,\cdot)](\theta)=m_{s-1}\dn^{s-2}\Big[-G_{s-1}(x,\cdot)+4k_{N,s}(s-1)\frac{\delta(x)^{s-1}\delta^{s-2}(1-|x|^2|\cdot|^2)}{[x,\cdot]^N}\Big](\theta)\nonumber\\
&=\frac{1}{\omega_N}\frac{\delta(x)^{s}}{|x-\theta|^N}=E_{m,s}(x,\theta),
\end{align*}
which implies \eqref{comp1}.  On the other hand, by Subsection \ref{remarks-derivative}, we have
\begin{align*}
\dn^{s-1}P_{s-1}(x,\cdot)&=\delta^{s-1}(x)\dn^{m}\Big[\delta^{m-1}\frac{(1-|x|^2|\cdot|^2)}{{[x,\cdot]}^N}\Big]\notag\\
&=\delta^{s-1}(x) \sum_{k=0}^{m}\dn^{k}\delta^{m-1}\dn^{m-k}\frac{(1-|x|^2|\cdot|^2)}{{[x,\cdot]}^N}=\delta^{s-1}(x) \dn^1\frac{(1-|x|^2|\cdot|^2)}{{[x,\cdot]}^N},
\end{align*}
where
\begin{align*}
\dn^{1}&\Big[\frac{(1-|\cdot|^2|x|^2)}{[x,\cdot]^N}\Big](\theta)
=-\lim_{y\to\theta}\frac{\partial}{\partial|y|^2}\Big[\frac{(1-|y|^2|x|^2)}{[x,y]^N}\Big]=-\frac{N}{2}\frac{\delta(x)}{|x-\theta|^{N+2}}-\frac{4-N}{4|x-\theta|^N}
\end{align*}
Hence, using \eqref{consta},
\begin{align*}
m_{s-1}\dn^{s-1}[-\Delta G_s(x,\cdot)](\theta)&= \frac{1}{\omega_N}\frac{\delta(x)^{s-1}}{|x-\theta|^N}+\frac{1}{2\omega_N}\delta(x)^{s-1}\frac{2|x|^2|x-\theta|^2+N\delta(x)(|x|^2-x\cdot \theta)}{ |x-\theta|^{N+2}}\notag\\
&=\frac{\delta(x)^{s-1}}{4\omega_N|x-\theta|^{N+2}}\Big( (4-N)\delta(x) |x-\theta|^2+N\delta(x)^2\Big),
\end{align*}
and \eqref{complement} follows by \eqref{expE}.
\end{proof}
\begin{proof}[Proof of Corollary \ref{hopf:lemma}]
By \cite[Theorem 1.1]{AJS16b}, we have that $u(x)=\int_B G_{s}(x,y)f(y)\ dy$ for $x\in \R^N$. Let $z\in\partial B$, then, by \eqref{comp1}, \eqref{comp1a}, and dominated convergence,
\begin{align*}
\frac{m_{s-1}}{m_s}\dn^{s-2} \Delta u(z) = \frac{m_{s-1}}{m_s}\int_B \dn^{s-2} \Delta (G_{s}(\cdot,y))(z)f(y)\ dy
=\int_B M_{s}(y,z)f(y)\ dy= \dn^s u(z).
\end{align*}
Equation \eqref{s:eq} follows similarly.
\end{proof}

\begin{remark}\label{conv:rem}
The convergence of the Green function $G_s$, the Martin kernel $M_s$, the Edenhofer kernels $E_{k,s}$, and the corresponding solutions as $\sigma\to 1^-$ is well behaved, in the sense that the (pointwise) limits exist and the resulting function is also a solution.  This can be easily verified (for suitable data) in virtue of Theorem \ref{main:thm} and the dominated convergence theorem (see also \eqref{comp1a}).  The convergence of the Poisson kernel $\Gamma_\sigma$ to the Poisson kernel for the Laplacian as $\sigma\to 1^-$ seems to be well known, but we could not find a precise reference (see \cite[footnote on page 121]{L72}).

Observe, nevertheless, that the limit as $\sigma\to 0^+$ may be more delicate for some kernels. We show this with a simple example: let $N=1$, $\sigma\in(0,1)$, $s=1+\sigma$, and $E_{m-1,s}$ given by \eqref{expE}; then
\begin{align*}
 u_s(x):=\int_{\partial B} E_{m-1,s}(x,y)\ dy=E_{m-1,s}(x,-1)+E_{m-1,s}(x,1)=(1-x^2)^{s-2}=(1-x^2)^{\sigma-1}
\end{align*}
is a solution (by Theorem \ref{main:thm}) of $(-\Delta)^s u_s = 0$ in $B$ satisfying $D^{\sigma-1} u_s(z) = 1$ and $D^{\sigma} u_s(z) = 0$ for $z\in\partial B$.
If $\sigma\to 0$ (\emph{i.e.}, if $s\to 1^+$), then $u_s(x)\to (1-x^2)^{-1}$, which is \emph{not} harmonic in $B$. Note that $u_1\not\in L^1(B)$ and that the extra boundary condition ($D^{\sigma-1} u_s(z) = 1$) required in the higher-order case ($s\in(1,2)$) is \emph{incompatible} with problems of lower order ($s=1$).  
 On the other hand, if $\sigma\to 1^-$ (\emph{i.e.}, if $s\to 2^-$), then $u_s(x)\to 1$ pointwisely for $x\in B$, which is, in fact, a solution of 
 \begin{align*}
 (-\Delta)^2 u_2 = 0\quad \text{ in $B$},\qquad D^{0} u_2(z) = u_2(z) = 1,\ \ D^{1} u_2(z) = \frac{1}{2}\partial_\nu 1 = 0\quad \text{ for }z\in\partial B.
 \end{align*}

\end{remark}

\begin{remark}\label{e:s:1sigma}
If $v\in \cL^1_{\sigma}\cap C^{\sigma+\alpha}(B_{r}\backslash\overline{B})$ for some $r>1$ and $f\in C^\alpha(\overline{B})$, $g\in C^{1+\alpha}(\partial B)$, and $h\in C(\partial B)$. Then the function $u:\R^N\to\R$ given by $u=v$ in $\mathbb R^N\backslash\overline{B}$ and
\begin{align*}
u(x)&:=\int_B G_{1+\sigma}(x,y)f(y)\ dy
+\int_{\R^N\backslash\overline{B}}\Gamma_\sigma(x,y) v(y)\ dy\\
&+\int_{\partial B} {E_{1,1+\sigma}}(x,z)\Big(h(z)-\gamma_{N,\sigma}\int_{\mathbb R^N\backslash\overline{B}}\frac{v(y)-v(z)}{(|y|^2-1)^{\sigma}|z-y|^N}\ dy\Big)\ dz
+\int_{\partial B}{E_{0,1+\sigma}}(x,z)g(z)\ dz\nonumber
\end{align*}
for $x\in B$,
is a pointwise solution of
 \begin{align}\label{claim}
(-\Delta)^{1+\sigma}u = f\ \text{ in }B,\quad u=v\ \text{ in }\R^N\setminus\overline{B},\quad \widetilde\dn^{\sigma-1}u = g\ \text{ on }\partial B,\quad \widetilde\dn^{\sigma}u = h\ \text{ on }\partial B,
\end{align}
where $\widetilde D$ is a suitable extension of the trace operator to functions with nonzero values at $\partial B$; this extension is given by
\begin{align*}
 \widetilde D^{\sigma-1}u(z)&=\lim_{\substack {x\to z\\ x\in B}}\delta(x)^{1-\sigma}(u(x)-\lim_{\substack{y\to z \\ y\in \R^N\backslash B}}u(y)),\\
 \widetilde D^{\sigma}u(z)&=\lim_{\substack {x\to z\\ x\in B}}\frac{\partial}{\partial|x|^2}\Big[\delta(x)^{1-\sigma}(u(x)-\lim_{\substack{y\to z \\ y\in \R^N\backslash B}}u(y))\Big],\\
\end{align*}
here the limits are always meant in the normal direction. Then, since $v\in \cL^1_{\sigma}\cap C^{\sigma+\alpha}(B_{r}\backslash\overline{B})$,
\begin{align}\label{long:comp}
{\varphi(z)}:={\gamma_{N,\sigma}}\int_{\mathbb R^N\backslash\overline{B}}\frac{v(y)-v(z)}{(|y|^2-1)^{\sigma}|z-y|^N}\ dy<\infty \qquad \text{ for }z\in\partial B\quad {\text{ and }\quad \varphi\in C(\partial B).}
\end{align}
This estimate can be shown by direct (but lengthy) computations and, to keep this paper short, we do not give the details here. {Assuming \eqref{long:comp}}, then
\begin{align*}
\widetilde \dn^{\sigma-1}\Bigg(\int_{\R^N\backslash\overline{B}}\Gamma_\sigma(x,y) v(y)\ dy-\int_{\partial B} {E_{1,1+\sigma}}(x,z){\varphi(z)}\ dy\ dz\Bigg)=0
\end{align*}
(by continuity) and
\begin{align*}
&\widetilde \dn^{\sigma} \Bigg(\int_{\R^N\backslash\overline{B}}\Gamma_\sigma(\cdot,y) v(y)\ dy-\int_{\partial B} {E_{1,1+\sigma}}(\cdot,w){\varphi(w)}\ dw\Bigg)(z)=\\
&={\varphi(z)}-m_{1+\sigma}\dn^{\sigma}\Bigg(\int_{\partial B} M_{1+\sigma}(\cdot,w){\varphi(w)}\ dw\Bigg)(z)={\varphi(z)}-\lim_{x\to \theta}\ch{\varphi(x)}={\varphi(z)}-{\varphi(z)}=0,
\end{align*}
where $\cH$ is the harmonic extension given in \eqref{GHHs}.  Therefore, \eqref{claim} follows from Theorem~\ref{thm:sHTOD}, \cite[Theorem 1.1]{AJS16b}, and \cite[Lemma 2.5]{FW16}.
\end{remark}

\appendix

\section{Existence of Green functions}

To show Theorem \ref{even:thm2} and for completeness, we guarantee in this appendix the existence of a Green function in general smooth domains. Although the strategy we follow here is standard, the higher-order fractional setting requires a special care; for instance, the pointwise definition of $(-\Delta)^s u=(-\Delta)^m(-\Delta)^\sigma u$ needs in particular that $u\in\cL^1_\sigma$, but the fundamental solution does not belong to this space in general (see below). To circumvent this difficulty, the use of distributional and weak solutions together with partial integration results under varying assumptions is necessary. For the reader's convenience we include the details here. Recall (\emph{e.g.}, from \cite[Theorem 5.10]{AJS16}) that, for any order $s>0$ and any dimension $N\in \N$, the \emph{fundamental solution} of $(-\Delta)^s$ is given by $F_{N,s}:\R^N\setminus\{0\}\to \R$, 
\begin{align}\label{FNs}
F_{N,s}(x):=\left\{\begin{aligned}
&
\frac{\Gamma(\frac{N}{2}-s)}{\pi^{\frac{N}{2}}4^{s}\Gamma(s)}\ \ |x|^{2s-N} && \quad\text{ for $s-\frac{N}{2}\not\in\N_0$,}\\
&\frac{2^{1-2s}(-1)^{s+1-\frac{N}{2}}}{\pi^{\frac{N}{2}}\Gamma(s-\frac{N}{2}+1)\Gamma(s)}
\ \ |x|^{2s-N}\ln|x| && \quad\text{ for $s-\frac{N}{2}\in\N_0$.}
\end{aligned}
\right.
\end{align}

\begin{lemma}\label{corrector-term}
Let $\Omega\subset \R^N$ be an open bounded set with smooth boundary. Then, for any $y\in \Omega$ there is a unique function $H(\cdot,y)\in H_{loc}^s(\R^N)\cap C^s(\R^N)\cap C^{\infty}(\Omega)$ solving
\begin{align}\label{H}
(-\Delta)^s H(\cdot,y)=0\quad \text{ in }\Omega,\qquad H(x,y)=F(x-y)\quad \text{ for }x\in \R^N\backslash \Omega
\end{align}
in the sense of distributions, that is, 
\begin{align*}
\int_{\R^N} H(x,y)(-\Delta)^s\phi(x)\ dx=0\qquad \text{ for all } \phi\in C^{\infty}_c(\Omega),\ y\in \Omega.
\end{align*}
Moreover, $H(x,y)=H(y,x)$ for $x,y\in \Omega$.
\end{lemma}
\begin{proof}
To ease notation, let $F=F_{N,s}$ and fix $y\in \Omega$ and $k\geq s$. Since $F$ is radially symmetric and $C^\infty(\R^N\backslash\{0\})$, there is $\zeta \in C^{2k+2{,0}}(\R^N)$ (depending on $y$) such that $\zeta=F(\cdot-y)$ in $\R^N\backslash \Omega$ and $\|\zeta\|_{C^{2k+2{,0}}(K)}<\infty$ for any compact set $K\subset \R^N$ (this follows directly from polynomial approximation of $F$ in a small neighborhood of $y$). Then $(-\Delta)^m\zeta\in \cL_\sigma^1\cap C^{2(k+1-m),0}(\R^N)$ and $f:=-(-\Delta)^\sigma(-\Delta)^m\zeta \in C^{2(k-m)+\alpha}(\R^N)$ for some $\alpha>0$, by an iteration of \cite[Proposition 2.7]{S07}. Furthermore, by \cite[Corollary 3.6]{AJS16a}, there is a unique weak solution $w\in \cH_0^s(\Omega)$ of $(-\Delta)^s w = f$ and $w\in C^s{(\R^N)}$, by elliptic regularity (see, \emph{e.g.}, \cite[Propositions 2.7 and 2.8]{S07} and \cite[Theorem 4]{G15:2}). Set $H(\cdot,y):= w+\zeta\in H_{loc}^s(\R^N)\cap C^s(\R^N)\cap C^{2k{,0}}(\Omega)$ and fix $\phi\in C^{\infty}_c(\Omega)$, then, by Lemma \ref{decay-s-smooth}, we find for every $t>0$ a constant $C(\varphi,t)>0$ such that
\begin{align*}
 |(-\Delta)^t\varphi(x)|\leq \frac{C(\varphi,t)}{1+|x|^{N+2t}}\qquad \text{ for all }x\in \R^N.
\end{align*}
Moreover, there is $c(\Omega,\zeta,k)>0$ such that $|(-\Delta)^k\zeta(x)|\leq c(\Omega,\zeta,k)|x|^{2(s-k)-N}$ for every $k\in\{0,\ldots,m\}$, in particular 
$|(-\Delta)^m\zeta(x)(-\Delta)^{\sigma}\phi(x)|\leq \widetilde C \frac{|x|^{2\sigma-N}}{1+|x|^{N+2\sigma}}$ for all $x\in\R^N$ and for some $\widetilde C>0$. 
Therefore, integrating by parts,
\begin{align*}
\int_{\R^N}\zeta(-\Delta)^{m}(-\Delta)^{\sigma}\phi\ dx=\int_{\R^N}(-\Delta)^{m}\zeta(-\Delta)^{\sigma}\phi\ dx=-\int_{\R^N}f\phi\ dx.
\end{align*}
By Lemma \ref{ibyp}, we have that $\int_{\R^N}w(x)(-\Delta)^{s}\phi(x)\ dx=\cE_{s}(w,\phi)$.  Thus
\begin{align*}
&\int_{\R^N}H(x,y)(-\Delta)^{s}\phi(x)\ dx=\cE_{s}(w,\phi)-\int_{\R^N}f(x)\phi(x)\ dx=0.
\end{align*}
Since $\varphi$ was taken arbitrarily, we have that \eqref{H} holds in the sense of distributions. We now argue uniqueness: let $v\in H^s_{loc}(\R^N)\cap C^{\infty}(\Omega)\cap C^s(\R^N)$ be another distributional solution of \eqref{H}. Then $v-H(\cdot,y)=0$ in $\R^N\setminus \Omega$ and, integrating by parts, $v-H(\cdot,y)\in \cH^s_{0}(\Omega)$ is $s$-harmonic in the weak sense. Then $v\equiv H(\cdot,y)$, by uniqueness of weak solutions (see \cite[Corollary 3.6]{AJS16a}).  Finally, since $F(x-y)=F(y-x)$ for all $x,y\in \R^N$, $x\neq y$, it follows that $H(x,y)=H(y,x)$, by uniqueness.
\end{proof}

Let $G_{\Omega}:\R^N\times \R^N\setminus \{(x,y)\in \Omega\times\Omega: x= y\}\to \R$ be given by 
\begin{align}\label{GOmega:def}
G_{\Omega}(x,y)=F_{N,s}(x-y)-H(x,y)\qquad \text{if $(x,y)\in (\Omega\times \R^N)\cup (\R^N\times\Omega)$}
\end{align}
and $G_{\Omega}(x,y)=0$ otherwise, where $H$ is the unique function given by Lemma \ref{corrector-term}.

\begin{prop}\label{existenceuniqueness}
Let $\Omega\subset\R^N$ be an open bounded set with smooth boundary. Then $G_\Omega$ is a Green function of $\Ds$ in $\Omega$.
 Moreover, if $f\in C^{\alpha}_c(\Omega)$, $\alpha\in(0,1)$, and 
\begin{align}\label{u7}
u=\int_{\Omega}G_{\Omega}(\cdot,y)f(y)\ dy,	 
\end{align}
then $u\in \cH^s_0(\Omega)$ is the unique weak solution of $(-\Delta)^su=f$ in $\Omega$ with $u=0$ on $\R^N\setminus \Omega$.	
\end{prop}
\begin{proof}
Clearly, $u=0$ on $\R^N \setminus \Omega$ when defined as in \eqref{u7}. By \eqref{GOmega:def}, we have that $u=u_1-u_2$ in $\R^N$, where
\begin{align*}
u_1(x):=\int_{\Omega} F_{N,s}(x-y)f(y)\ dy\qquad \text{ and }\qquad u_2(x):=\int_{\Omega} H(x,y)f(y)\ dy\qquad \text{ for }\ x\in \R^N. 
\end{align*}
By \cite[Corollary 5.16]{AJS16}, we have that $u_1$ is a distributional solution of $(-\Delta)^s u_1=f$ in $\R^N$, \emph{i.e.},
\begin{align}\label{u1:dist}
 \int_{\R^N} u_1(x)(-\Delta)^s\phi(x)\ dx=\int_{\Omega} f(x)\phi(x)\ dx\qquad \text{ for all }\phi\in C^{\infty}_c(\Omega).
\end{align}
Moreover, by iterating \cite[Lemma 5.9]{AJS16}, there is a polynomial $R_s$ of degree at most $2m-N$ such that $(-\Delta)^m u_1=u_{11} + u_{12}$ in $\R^N$ with 
\begin{align*}
u_{11}=F_{N,\sigma}\ast f=\int_{\R^N}F_{N,\sigma}(\cdot-y)f(y)\ dy\qquad \text{ and }\qquad u_{12}=\int_{\Omega}R_s(\cdot-y)f(y)\ dy.
\end{align*}
Then $u_{12}\in C^{\infty}(\R^N)$ and $u_{11}\in C^{2\sigma+\alpha}(\R^N)$ (see \cite{L72} or \cite[Proposition 2.8]{S07} using that $F_{N,\sigma}\ast f\in L^{\infty}(\R^N)$, because $f\in C^{\alpha}_c(\Omega)$). In particular, $(-\Delta)^mu_1 \in C^{2\sigma+\alpha}(\R^N)$, which then implies that $u_1\in C^{2s+\alpha}(\R^N)\subset H^{2s}_{loc}(\R^N)$.
On the other hand, since $f$ has compact support, Lemma \ref{corrector-term} implies that $u_2\in C^{\infty}(\Omega)\cap H_{loc}^s(\R^N)$. Hence, $u\in H_{loc}^s(\R^N)$ and since $u=0$ in $\R^N\backslash\overline{B}$, we have that $u\in\cH_0^s(\Omega)$ and, by Lemma \ref{ibyp}, it suffices to show that
\begin{align}\label{goal:dist}
 \int_{\R^N} u(x)(-\Delta)^s\phi(x)\ dx =\int_{\Omega} f(x)\phi(x)\ dx \qquad \text{ for all }\phi\in C^{\infty}_c(\Omega).
\end{align}
Since $R_s$ is a polynomial of degree at most $2m-N$ we have, by Lemma~\ref{decay-s-smooth}, Fubini's theorem, and Lemma \ref{corrector-term}, that
\begin{align}\label{u2:dist}
\int_{\R^N} u_2(-\Delta)^s\phi\ dx = \int_{\Omega} f(y) \int_{\R^N} H(x,y)(-\Delta)^s\phi(x)\ dx\ dy=0\qquad \text{ for all }\phi\in C^{\infty}_c(\Omega).
\end{align}
Thus \eqref{goal:dist} follows from \eqref{u1:dist} and \eqref{u2:dist}, and the proof is finished.
\end{proof}

\begin{remark}\label{pointwise:rmk}
The pointwise definition of $(-\Delta)^s$ can be a delicate issue.  To be more precise, let $x\in \R^N$, $\sigma\in(0,1)$, $m\in\N$ even, $s=m+\sigma$, and consider the following three options
\begin{align*}
 (i)\ \  (-\Delta)^m(-\Delta)^\sigma u(x),\qquad (ii)\ \  (-\Delta)^\frac{m}{2} (-\Delta)^\sigma (-\Delta)^\frac{m}{2}u(x),\qquad (iii)\ \  (-\Delta)^\sigma(-\Delta)^m u(x).
\end{align*}
The adequacy of each of these alternatives depends on the problem and the properties of the solutions, in particular, on the \emph{global} regularity and the growth at infinity. A first observation is that the above three options require to be at least of class $C^{2s}$ at $x$. However, $(ii)$ and $(iii)$ additionally need a global regularity assumption such as $u\in C^{m{,0}}(\R^N)$ and $u\in C^{2m{,0}}(\R^N)$ respectively. This already restricts the kind of solutions that can be studied with definition $(iii)$. For example, consider a weak solution $v\in\cH_0^s(B)$ of the problem 
\begin{align}\label{wprob}
(-\Delta)^s v = f\quad \text{ in }B,\qquad v= 0\quad \text{ on }\R^N\backslash\overline{B}, 
\end{align}
where $f\in C^\alpha(\overline{B})$, $\alpha>0$. By regularity \cite{AJS16}, $v\in C^{s}(\R^N)$ and this is in general optimal; for instance, $\delta^s$ does not belong to $C^{s+\varepsilon}(\R^N)$ for any $\varepsilon>0$ and $\delta^s$ is a weak solution of \eqref{wprob} with $f\equiv c$ for some constant $c>0$ (see \cite[Lemma 2.2]{RS15} and Lemma \ref{ibyp}). 

Observe that $(ii)$ is the pointwise notion suggested by the scalar product \eqref{scalar:p} in $\cH_0^s(B)$. Moreover, under suitable assumptions (see \cite[Proposition B.2]{AJS16}), one can interchange some derivatives with the fractional Laplacian $(-\Delta)^\sigma$ and, in particular, $(i)$ and $(ii)$ coincide for weak solutions of \eqref{wprob}.

However, definition $(i)$ is the most restrictive in terms of \emph{growth} at infinity, since to compute $(-\Delta)^\sigma u$ a condition such as $u\in\cL^1_\sigma$ is needed. Thus, if $s$ is large with respect to $N$ (for example, if $2m\geq N$), one cannot apply $(-\Delta)^m(-\Delta)^\sigma$ to the fundamental solution $F_{N,s}$ (see \eqref{FNs}) or to the function $\zeta$ in the proof of Lemma~\ref{corrector-term}.  In this case, $(iii)$ is more adequate. Indeed, if $\sigma\neq \frac{1}{2}$ then $-\Delta F_{N,s}=F_{N,s-1}$ (\emph{cf.} \cite[Lemma 5.9]{AJS16}). Thus $(-\Delta)^m F_{N,s}=F_{N,\sigma}\in \cL^1_\sigma\cap C^\infty(\R^N\backslash\{0\})$ and $(-\Delta)^\sigma(-\Delta)^m F_{N,s}(x)$ can be computed at any $x\in \R^N\backslash \{0\}$.

A more complicated phenomenon happens in Corollary~\ref{Ls}, where neither of the above definitions can be applied: consider some outside data $\psi\in \cL^1_s\backslash \cL_{s-1}^1$ with $\psi=0$ in $B_{r}$, $r>1$; then the function $w:=\cH_s \psi\in C^{\infty}(B)\cap C^s(B_r)\cap \cL^1_s$ is $s$-harmonic in the sense of distributions (\emph{i.e.} \eqref{dist:sol} holds); nevertheless, because of its growth at infinity, one cannot compute $(i)$ nor $(ii)$, and even $(iii)$ may fail, since $(-\Delta)^m \psi$ may not be in $\cL^1_\sigma$.  Furthermore, $w$ does not belong in general to $C^{m+1{,0}}(\R^N)$ even if $\psi\in C^\infty(\R^N)$ (see \emph{e.g.} \eqref{u:dec}, where $\varphi>0$ in $\overline{B}$ if $\psi$ is nonnegative and nontrivial, which would imply that $w\in C^s(\R^N)\backslash C^{s+\varepsilon}(\R^N)$ for any $\varepsilon>0$).

Observe that the pointwise notion of $(-\Delta)^s$ is not relevant to define solutions in the sense of distributions, since all derivatives commute with $(-\Delta)^\sigma$ for functions in $C_c^\infty(\R^N)$ (see \cite[Proposition B.2]{AJS16}). 

Finally, we mention that there is a more involved pointwise evaluation of $(-\Delta)^s$ in terms of a hypersingular integral with finite differences that can be applied directly in all the cases described above, \emph{i.e.}, for functions which are only required to be locally of class $C^{2s}$ and in $\cL_s^1$; we refer to \cite{AJS17b} for details.

\end{remark}

\subsection*{Acknowledgments}

A. Salda\~{n}a was supported by the Alexander von Humboldt Foundation (Germany) and by the Instituto Superior T\'{e}cnico (Portugal).  N. Abatangelo and A. Salda\~{n}a thank the hospitality of the Goethe-Universit\"at, Frankfurt, where this research started during a short visit. The authors thank Tobias Weth for his encouragement and Gerd Grubb for helpful discussions.

%

\end{document}